\documentclass[a4paper,12pt]{amsart}
\usepackage[utf8]{inputenc}
\usepackage[T1]{fontenc}
\usepackage[UKenglish]{babel}
\usepackage[margin=18mm]{geometry}

\usepackage{color}%cyan,red;magenta,green;yellow,blue

\usepackage{appendix}

\usepackage{graphicx}

\usepackage{amsmath,amssymb,amsfonts,amsthm}
\usepackage{mathrsfs,eucal,dsfont}
\usepackage{verbatim,enumitem}

\usepackage{hyperref,url}

\newcommand{\R}{\mathds R}

\newcommand{\I}{\mathds 1}

\def\d{{\rm d}}
\def\<{\langle}
\def\>{\rangle}

 \def\ss{\sqrt}

\def\R{\mathbb R}   \def\ss{\sqrt} 
 \def\kk{\kappa} 
  \def\vv{\varepsilon} 
\def\<{\langle} \def\>{\rangle}  
  \def\nn{\nabla}  
\def\d{\text{\rm{d}}}   
  \def\si{\sigma} 
 \def\beq{\begin{equation}}  
 
\def\e{\text{\rm{e}}}    
  
 \def\P{\mathbb P}

\def\E{\mathbb E}

\def\to{\rightarrow}
\def\8{\infty}\def\3{\triangle}
\def\1{\lesssim}

\renewcommand{\bar}{\overline}
\renewcommand{\hat}{\widehat}
\renewcommand{\tilde}{\widetilde}

\newtheorem{theorem}{Theorem}[section]
\newtheorem{lemma}[theorem]{Lemma}
\newtheorem{proposition}[theorem]{Proposition}
\newtheorem{corollary}[theorem]{Corollary}

\theoremstyle{definition}

\newtheorem{example}[theorem]{Example}
\newtheorem{remark}[theorem]{Remark}

\renewcommand{\theequation}{\thesection.\arabic{equation}}
\numberwithin{equation}{section}
\begin{document}
\allowdisplaybreaks

\title[existence, uniqueness, and multiplicity] {Stationary distributions of McKean-Vlasov SDEs with jumps: existence, uniqueness, and multiplicity}

\author{
Jianhai Bao  \qquad
Jian Wang}
\date{}
\thanks{\emph{J.\ Bao:} Center for Applied Mathematics, Tianjin University, 300072  Tianjin, P.R. China. \url{jianhaibao@tju.edu.cn}}

\thanks{\emph{J.\ Wang:}
School  of Mathematics and Statistics \& Key Laboratory of Analytical Mathematics and Applications (Ministry of Education) \& Fujian Provincial Key Laboratory
of Statistics and Artificial Intelligence, Fujian Normal University, 350007 Fuzhou, P.R. China. \url{jianwang@fjnu.edu.cn}}

\maketitle

\begin{abstract}
In this paper, we are interested in the issues on existence, uniqueness, and multiplicity of stationary distributions
for McKean-Vlasov SDEs with jumps. In detail, with regarding to  McKean-Vlasov SDEs driven by pure jump L\'{e}vy processes,
we principally    (i)  explore  the existence of stationary distributions via
Schauder's fixed point theorem under an appropriate Lyapunov condition;  (ii) tackle the uniqueness of stationary distributions and the convergence to the equilibria
as long as the underlying drifts are continuous with respect to the measure variables under the weighted total variation distance
and the $L^1$-Wasserstein distance, respectively; (iii) demonstrate  the multiplicity of stationary distributions under a locally dissipative condition. In addition, some illustrative examples are provided to show that the associated McKean-Vlasov SDEs possess a unique,  two  and three stationary distributions, respectively.

\medskip

\smallskip

\end{abstract}
\section{Background and main results}
\subsection{Background}
For a linear Markov process  (i.e., the transition kernel depends merely on the current state), there are various methods to show  existence of invariant probability measures (IPMs for short); see, for instance,  \cite{DZ,Hairer} via  Krylov-Bogoliubov's criterion, \cite{PR} by the aid of the remote start approach, and \cite{HMS,Eb,LWa} on the basis of   weak contractivity under some  probability (quasi)-distances, to name only a very few of them. With regards to the previous  methods,
the essential ingredient is that the Markov operator generated by the underlying Markov process is a linear Markov semigroup. Nevertheless, concerning  a nonlinear Markov process (i.e.,   the  transition function  may depend not only on the current state  but also on the current
distribution of the process), the strategies mentioned previously are no longer available  to tackle existence of IPMs,
which  are also termed as stationary distributions in literature (throughout this paper, we tend to call them stationary distributions); see e.g. \cite{MR, Zhang}.
In particular, the nonlinear property of the corresponding
stochastic process results in the invalidity of the previous  techniques, which are  applied to handle existence of stationary distributions for linear Markov processes.

In the real world, the distribution of a stochastic process under investigation embodies  an macroscopic quantity. Due to the involvement of distribution variables, the evolution behaviour of the resulting stochastic systems might  change dramatically. As a toy example,
we consider the scalar Ornstein–Uhlenbeck process $(X_t)_{t\ge0}$,  which is determined by  the   linear SDE:
\begin{align}\label{1E}
\d X_t=-\lambda X_t\,\d t+\d W_t,
\end{align}
where $\lambda>0$ and $(W_t)_{t\ge0}$ is a $1$-dimensional Brownian motion.  It is very well known that the SDE \eqref{1E} possesses a unique IPM.
Once the expectation of the solution process (i.e., $\E X_t$) is plugged into the drift part, we   correspondingly derive the following mean-field SDE:
\begin{align}\label{2E}
\d X_t=\big(-\lambda X_t+\E X_t\big)\,\d t+\d W_t.
\end{align}
As shown in \cite[p.\ 403--404]{AD}, the SDE \eqref{2E} enjoys a unique stationary distribution (which indeed is a zero-mean Gaussian measure)
in case of $\lambda\neq1;$ nonetheless, for the setting $\lambda=1$,
the SDE \eqref{2E} admits infinite many stationary distributions by solving a self-consistent algebraic equation. Therefore, as far as the mean-field SDE \eqref{2E} is concerned, the phenomenon on phase transition has taken place by virtue of the intervention of distribution arguments.

Below, as for the topic on phase transition of distribution dependent SDEs (which, in terminology,  are also named as mean-field or McKean-Vlasov SDEs),
we give a rough  overview on related progresses. For  convenience of presentation,   we consider the following  mean-field SDE on $\R$:
\begin{align}\label{3E}
\d X_t=-\bigg(V'(X_t)+\int_{\R}F'(\cdot,y)(X_t)\mathscr L_{X_t}(\d y)\bigg)\d t +\si\d W_t,
\end{align}
where the confinement potential $V:\R\to\R$, the interaction potential $F:\R\to\R$, $\mathscr L_{X_t}$ means the law of $X_t,$ $\si$ is a non-zero constant, and $(W_t)_{t\ge0}$ is a $1$-dimensional Brownian motion.
In the seminal work \cite{Dawson}, Dawson worked on the scalar SDE \eqref{3E} with $V(x)= \frac{1}{4}x^4-\frac{1}{2}x^2 $ and $F(x,y)=\frac{1}{2}\theta(x-y)^2$ for some constant $\theta>0$, and revealed that the corresponding SDE has a unique stationary distribution when $\si$ is greater than or equal to some positive threshold $\si_c$,
and three stationary distributions once $\si\in(0,\si_c)$. Later on, \cite{Tugaut}
considered the SDE \eqref{3E} with a much more general confinement potential, in which  $V$ is a polynomial function with the degree $\mbox{deg}(V)\ge4$  and $F(x,y)=\frac{1}{2}\theta(x-y)^2$ for some constant $\theta>0$. Satisfactorily, a positive threshold
$\si_c$ was also provided therein and it was shown that the associated SDE \eqref{3E} possesses a unique symmetric stationary distribution as long as $\si\ge \si_c$, and exactly three stationary distributions in case of $\si\in(0,\si_c).$ No matter what \cite{Dawson} or \cite{Tugaut}, given the intensity  $\theta$  of the interaction potential, the intensity  $\si$  of the noise can lead to an occurrence of phase transition.
Recently,  \cite{LXZ} investigated \eqref{3E} with $V(x)=x^4-\beta x^2 $ for some $\beta>0$,
$F(x,y)=-\gamma xy $ for some $\gamma>0$, and $\si=\ss2$. In particular,
 when the noise intensity was given in advance, \cite{LXZ} showed that there exist critical values $\gamma_c=2\ss 3$ and $\beta_c=(12-\gamma^2)/(2\gamma)>0$ such that the SDE \eqref{3E} has $(i)$  three stationary distributions as soon as $\gamma\in(0,2\ss3)$ and $\beta>\beta_c$, or $\gamma\ge2\ss3$ and $\beta>0;$ (ii) a unique stationary distribution whenever $\gamma\in(0,2\ss3)$
and $\beta\in(0,\beta_c).$
With contrast to  \cite{Dawson,Tugaut}, \cite{LXZ} brought the phenomenon on phase transition into light based on another aspect. More precisely,
the authors in \cite{LXZ} demonstrated that the phase transition arises from the inverse temperature $\beta$ and the interaction intensity $\gamma$, rather than the noise intensity $\si.$

Concerning the previous references   \cite{Dawson,LXZ,Tugaut}, there is a common feature. More precisely,  the McKean-Vlasov SDE under investigation enjoys a special structure (i.e., it is of gradient type) so that the number of stationary distributions can be tracked by analyzing in depth the number of roots to a self-consistent algebraic equation. In the past few years, the research on existence, uniqueness as well as non-uniqueness of stationary distributions has advanced greatly for much more general McKean-Vlasov SDEs. In particular, \cite{Butkovsky} and \cite{Wang} showed that the McKean-Vlasov SDE under consideration  has a unique stationary distribution when the law of the solution process is weakly contractive under the weighted total variation distance and the $L^2$-Wasserstein distance, respectively. With the aid of the generalized Banach fixed-point theorem for multivalued maps, \cite{AD} discussed the existence of stationary distributions  for McKean-Vlasov SDEs, where the drift term and the diffusion term are Lipschitz continuous with respect to the spatial variable and the measure variable. Recently, there is also considerable  progress on  the
 non-uniqueness of stationary distributions for McKean-Vlasov SDEs without a specific framework. Particularly,  via Schauder's fixed point theorem,
  \cite{Zhang} furnished sufficient criteria to study the non-uniqueness of stationary distributions for multi-dimensional McKean-Vlasov SDEs.
Under a locally dissipative condition, by means of the  preservation of stochastic order, \cite{QYZ} exploited the multiplicity of stationary distributions for $1$-dimensional McKean-Vlasov SDEs.
As mentioned previously, we know that there is a huge amount of work devoted to
 stationary distributions (including existence, uniqueness, and non-uniqueness)
for  McKean-Vlasov SDEs driven by Brownian
motion.

 Besides the macroscopic influence (e.g., the distribution element),   in certain scenarios,
  a system under consideration might be subject  to  random fluctuations   with discontinuous sample paths instead of continuous
 counterparts. So, it is quite natural and reasonable to consider McKean-Vlasov SDEs driven by pure jump processes. In contrast to McKean-Vlasov SDEs with Brownian motion,
 the exploration on existence, uniqueness, and non-uniqueness associated with McKean-Vlasov SDEs with jumps
 is still in the initial stage and so rather incomplete. To motivate    our present work, we consider the following illustrative
example:
\begin{align}\label{4E}
\d X_t=b(X_t,\mathscr L_{X_t})\,\d t+\d Z_t,
\end{align}
where $(Z_t)_{t\ge0}$ is a   symmetric (rotationally invariant) $\alpha$-stable process in $\R$ with $\alpha\in(1,2)$, and
\begin{align*}
b(x,\mu):=-C_{2-\alpha}\e^{V(x,\mu)}\int_\R\frac{\e^{-V(y,\mu)}V'(\cdot,\mu)(y)}{|x-y|^{\alpha-1}}\,\d y,\quad x\in\R,\,\mu\in\mathscr P_1.
\end{align*}
Hereinabove, for $r\in(0,1)$ and
 $\gamma,\beta>0,$
 $$C_{r}:=\frac{\Gamma((1-r )/2)}{{2^{r}\pi^{1/2}\Gamma(r/2)}}\quad \mbox{ and } \quad
V(x,\mu) :=\gamma x\int_\R y\mu(\d y)-x^4+\beta x^2.  $$
In terms of \cite[Proposition 3.3]{HMW}, $\pi(\d x)=\e^{\gamma m x -x^4+\beta x^2}/C_V\,\d x$ is a stationary distribution of the SDE \eqref{4E}, in which $C_V$ is a normalization constant and $m=\int_\R y \pi(\d y)$. In addition, the following self-consistent  algebraic equation:
\begin{align*}
\int_\R(x-m)\e^{\gamma m x -x^4+\beta x^2}\,\d x=0
\end{align*}
holds true. Thus, by invoking \cite[Theorem 2(1)]{LXZ}, we can conclude that there exists a threshold $\beta_c>0$ such that the SDE \eqref{4E} admits (i) a unique stationary distribution as soon as $\beta\in(0,\beta_c)$; (ii) exactly three stationary distributions in case of $\beta\ge\beta_c$. The previous example convinces us that,  just like McKean-Vlasov SDEs driven by Brownian
motion, the phase transition might occur as far as  McKean-Vlasov SDEs driven by the pure jump L\'{e}vy noise are concerned.  Whilst, the systematic study on existence, uniqueness and non-uniqueness
of McKean-Vlasov SDEs with pure jumps is still vacant. This encourages us to move on and initiate such an uncultivated  subject.

When a  McKean-Vlasov SDE  has a unique stationary distribution, the convergence to the equilibrium  has been widely and intensively investigated in   e.g.  \cite{BSWX,Butkovsky,EGZ,LMW,LWZ,Wangb}.
 However, the corresponding issue will become extremely tough once the underlying McKean-Vlasov SDE has multiple stationary distributions. For the granular media equation with quadratic interaction,
\cite{MR}  provided quantitative convergence rates for initial conditions in a Wasserstein
ball around the stationary solutions via  non-linear versions of the local log-Sobolev inequalities. By linearizing the nonlinear Markov semigroup generated by the distribution flow
associated with the McKean-Vlasov SDE, the local convergence near equilibria was treated in \cite{Zhangb}. For further developments on the convergence to the equilibria, we would like to mention e.g. \cite{Cormier,QYZ,Tugautb,Tugautc} and the references within. Nevertheless,  for L\'{e}vy-driven McKean-Vlasov SDEs with multiple stationary distributions,
the topic on the local convergence to the equilibria has been sparsely investigated. The present work will undoubtedly pave the way for us to carrying out this topic in the future. This can be regarded as another motivation for  our present work.

\subsection{Main results}

Concretely speaking,
in this paper we work on the following McKean-Vlasov SDE with jumps:
\begin{align}\label{E1}
\d X_t=b(X_t,\mathscr L_{X_t})\,\d t+\d Z_t,
\end{align}
where $b:\R^d\times\mathscr P(\R^d)\to\R^d$, and $(Z_t)_{t\ge0}$ is a $d$-dimensional pure jump L\'{e}vy process with the L\'{e}vy measure $\nu(\d z)$
satisfying the following integrable condition: for some $\beta_*\in(0,2],$
\begin{align}\label{E2}
\nu\big( |\cdot|^2\I_{\{|\cdot|\le1\}}\big) +\nu\big(|\cdot|^{\beta_*}\I_{\{|\cdot|>1\}}\big)  <\8.
\end{align}
 Hereinbefore, $\mathscr P(\R^d)$ denotes  the set of probability measures on $\R^d$, and $\nu(f):=\int_{\R^d}f(x)\nu(\d x)$ for a $\nu$-integrable function $f:\R^d\to\R$.
Throughout the paper, we assume that the McKean-Vlasov SDE \eqref{E1} is strongly well-posed under suitable conditions; see, for example,  \cite[Theorem 1]{BLW} and references within for related details.

\subsubsection{{\bf Existence of stationary distributions}}
First of all, we introduce some  notations to be used frequently.
For $p>0$, let $\mathscr P_p(\R^d)=\{\mu\in\mathscr P(\R^d): \mu(|\cdot|^p)<\8\}$, i.e., the collection of probability measures on $\R^d$ with finite $p$-th moment. Below, we shall   write simply $\mathscr P_p$ instead of $\mathscr P_p(\R^d)$ once the state space $\R^d$ is unimportant.  Set $\mathscr P_p^M:=\{\mu\in\mathscr P_p:\mu(|\cdot|^p)\le M\}$ for given $M,p>0$.

To address existence of stationary distributions of \eqref{E1}, the following assumptions are in force.
\begin{enumerate}\it
\item[{\rm(${\bf A}_1$)}]  there exist constants $\beta\in (0,\beta_*]$, $\lambda_1, \lambda_2, C_b, \theta_i >0,\, i=1,3,4,$ and $\theta_2\ge0$ satisfying that  $\theta_1\ge1-\beta/2$, $\theta_2<1+\theta_1$,  $\gamma_1\in[0,1)$,    $\theta_3\in(0,\beta^*]$, and
\begin{align*}
(i)\quad  \beta^*(1-\gamma_1)>\theta_3\theta_4; \quad \mbox{ or } \quad (ii)  \quad \beta^*(1-\gamma_1)=\theta_3\theta_4 \,\, \mbox{ and } \,\, \lambda_1>\lambda_2,
\end{align*}
such that for all $x\in\R^d$ and $\mu\in\mathscr P_{\beta^*}$,
\begin{align}\label{E12}
\<x,b(x,\mu)\>\le C_b-\lambda_1|x|^{1+\theta_1} +\lambda_2 (1+|x|^2)^{\theta_2/2}\mu(|\cdot|^{\theta_3})^{\theta_4},
\end{align}
where  $\beta^*:=\beta+\theta_1-1$  and
 $\gamma_1:= (\beta+\theta_2-2)^+/\beta^*.$

\item[{\rm(${\bf A}_2$)}]for fixed $\mu\in \mathscr P_{\beta^*}$, $(Y_t^{\mu})_{t\ge0}$ solving
the  time-homogenous  SDE:
\begin{align}\label{EW5}
 \d Y_t^{\mu}=b(Y_t^{\mu},\mu)\,\d t+\d Z_t,\quad t>0
 \end{align}
is a $C_b$-Feller process and has a unique  IPM  $\pi^\mu$.

\item[{\rm(${\bf A}_3$)}] for some constant  $M^*>0$, $ \mathscr P_{\beta^*}^{M^*} \ni \mu\mapsto \pi^\mu$ is weakly continuous.
\end{enumerate}

 Our first main result in this paper is as follows.

\begin{theorem}\label{thm1}
Under Assumptions $({\bf A}_1)$, $({\bf A}_2)$   and  $({\bf A}_3)$, \eqref{E1} has a stationary distribution $\pi$ with finite $\beta^*$-th moment $($i.e.,
$\pi(|\cdot|^{\beta^*})<\8$$)$.
\end{theorem}

 \begin{remark}\label{Rem1.2}  We make some comments on Assumptions $({\bf A}_1)$, $({\bf A}_2)$   and  $({\bf A}_3)$.
 \begin{itemize}
 \item[{\rm(1)}] \eqref{E12} is the Lyapunov condition, which, along with Assumption (${\bf A}_2$),  plays a crucial role in showing that the mapping $\Lambda_\cdot$ defined in \eqref{S10} below leaves invariantly  a suitable subset $\mathscr P^{M^*}_{\beta^*}$ of $\mathscr P_{\beta^*}$, where the quantity $M^*>0$   given in Assumption (${\bf A}_3$)  has an explicit expression; see  Lemma \ref{lem2} below.
   \item[{\rm(2)}] The dissipative property in \eqref{E12} comes from the negative term $-\lambda_1|x|^{1+\theta_1}$, and so the term $\lambda_2 (1+ |x|^2)^{\theta_2/2}\mu(|\cdot|^{\theta_3})^{\theta_4}$ can be regarded as a lower-order perturbation at least when $|x|$ is  large enough. In this sense, the prerequisite $\theta_2<1+\theta_1$ is reasonable.
       By applying It\^o's formula to $V_\beta(x):=(1+|x|^2)^{{\beta}/{2}}$ for all $x\in\R^d$, we obtain from \eqref{E12} that the   terms $-\lambda_1\beta(1+|x|^2)^{\frac{1}{2}(\beta-2)}|x|^{1+\theta_1}$ and $ \lambda_2\beta(1+|x|^2)^{\frac{1}{2}(\beta+\theta_2-2)} \mu(|\cdot|^{\theta_3})^{\theta_4}$ will show up, where the former determines  the moment order (i.e., $\beta^*=\beta+\theta_1-1$)
of  the IPM, and the latter one can be dominated by the linear combination of $V_{\beta^*}(x)$ and $\mu(|\cdot|^{\beta^*})^{{\theta_3\theta_4}/({\beta^*(1-\gamma_1)})}$ by making use of the Young inequality. To show that a suitable subspace of $\mathscr P_{\beta^*}$ is invariant, the order  $\frac{\theta_3\theta_4}{\beta^*(1-\gamma_1)}$ should be less than or equal to $1$. Based on this, $(i)$ or $(ii)$ is imposed, where, for the critical case, the additional condition $\lambda_1>\lambda_2$ is necessary.

 \item[{\rm(3)}]
The  examinations of Assumptions (${\bf A}_2$) and (${\bf A}_3$) will be detailed in Section \ref{sec2}. In addition, Assumption (${\bf A}_3$)
is enforced for the applicability of Schauder's fixed point theorem, which  is the main tool
 to prove Theorem \ref{thm1}.
Generally speaking, the statement that the mapping $\mathscr P_{\beta^*}\ni\mu\mapsto \pi^\mu$ is
 weakly continuous is not true as shown by Example \ref{exm} since, in general,  the weak convergence does not imply the convergence under the Wasserstein distance.
 Nevertheless, Assumption $({\bf A}_3)$  is valid when $\mathscr P_{\beta^*}$ is replaced by the subset  $\mathscr P_{\beta^*}^{M^*}$.\end{itemize}
\end{remark}

\subsubsection{{\bf Multiplicity of stationary distributions}}
In this subsection, with regarding to \eqref{E1},
 we turn to focus on the issue on existence of multiple stationary distributions,   and, most importantly,  aim at
providing some explicit  conditions to judge the multiplicity.

In addition to  Assumptions $({\bf A}_1)$, $({\bf A}_2)$ and  $({\bf A}_3)$, we further suppose that
\begin{enumerate} \it
\item[{\rm(${\bf A}_4$)}]there exist
 $y\in\R^d$, $\vv>0$,
  and a continuous  function $g_{\vv}:[0,\8)\times[0,\8)\to\R$    such that for all $x\in\R^d$ and  $\mu\in\mathscr P_{\beta^*} $,
\begin{equation}\label{S6}
\begin{split}
&\beta(\vv+ |x-y|^2)^{\frac{1}{2}(\beta-2)} \<x-y,b(x,\mu)\>
+ 2^{\frac{\beta}{2}}\nu(|\cdot|^{\frac{\beta}{2}}\I_{\{|\cdot|>1\}})|x-y|^{\frac{\beta}{2}}\\
&\quad+\frac{1}{2}\beta\vv^{\frac{1}{2}(\beta-2)}
\nu(|\cdot|^2\I_{\{|\cdot|\le 1\}})+\nu(|\cdot|^{\beta}\I_{\{|\cdot|>1\}})\\
&\le -g_{\vv}(|x-y|,\mu(|y-\cdot|)),
\end{split}
\end{equation} where $\beta\in (0,\beta_*]$ is given in Assumption {\rm(${\bf A}_1$)}.
Furthermore, we assume that, for each fixed $r_1\in[0,\8) $,   $r_2\mapsto g_{\vv }(r_1,r_2)$ is decreasing; for fixed $r_2\in[0,\8)$,  $r_1\mapsto g_{\vv }(r_1,r_2)$ is strictly convex; and for some $M_*>0$,
$g_{\vv}(r,M_*)>0$ for all $r\ge M_*$.
\end{enumerate}

Under Assumptions $({\bf A}_1)$, $({\bf A}_2)$ and $({\bf A}_3)$, Theorem \ref{thm1} has shown us   that \eqref{E1} has a stationary distribution
in a suitable subspace $\mathscr P_{\beta^*}^{M^*}$.  If  Assumption  $({\bf A}_4)$ is imposed additionally, it will be revealed  that  \eqref{E1} might have a stationary distribution in a subspace of  $\mathscr P_{\beta^*}^{M^*}$. As indicated  in \cite[Proposition 1.2]{Tugaut} which is  concerned with McKean-Vlasov SDEs driven by Brownian
motion, stationary distributions  could concentrate on the local minimizers of confinement  potentials  involved.
In the spirit of \cite[Theorem 3.1]{Zhang}, we intend to prove  that each stationary distribution of \eqref{E1} should be totally distinct when the distance of centers for each stationary distribution
  with proper concentration is large.
The aforementioned assertion is formulated  as   the following theorem.

\begin{theorem}\label{thm3}
Assume that Assumptions $({\bf A}_1)$, $({\bf A}_2)$, $({\bf A}_3)$ and $({\bf A}_4)$ hold. Then, we have the following statements:
\begin{itemize}
\item[{\rm(i)}]  \eqref{E1} has a stationary distribution in
$\mathscr P_{\beta^*,y}^{ M^*,M_*}: = \{\mu\in\mathscr P_{\beta^*}^{M^*} :\mu(|y-\cdot|)\le M_*\} $,  where $M^*>0 $ is to be given explicitly  in Lemma $\ref{lem2}$ below, and $M_*>0$ is stipulated in $({\bf A}_4)$;

\item[{\rm(ii)}] If
$({\bf A}_4)$ holds true respectively for some  $y_i\in\R^d$, $1\le i\le k$, and $$ M_*<\frac{1}{4}\min_{1\le i\neq j\le k}|y_i-y_j|,$$ then \eqref{E1} has at least $k$-distinct stationary distributions in
$\mathscr P_{\beta^*, y_i}^{ M^*,M_*}, 1\le i\le k$,  respectively.
\end{itemize}
\end{theorem}

\begin{remark} We make some comments on Assumption $({\bf A}_4)$, which in terminology  is termed as the locally  dissipative condition when the noise process $(Z_t)_{t\ge0}$ is a $d$-dimensional Brownian motion; see, for instance, \cite[Definition 2.3]{QYZ}.
\begin{itemize}
\item[{\rm (1)}]
Since, for given $\beta\in(0,\beta_*]\subset (0,2]$ and $y\in\R^d$,  $\R^d\ni x\mapsto V_{\beta,y}(x):=|x-y|^\beta$ is not a $C^2$-function, we modify  it to  a smooth one $V_{\vv,\beta,y}(x):=(\vv+|x-y|^2)^{\frac{1}{2}\beta}$ so that the It\^o formula can be applied to   $(V_{\vv,\beta,y}(Y_t^\mu))_{t\ge0}$. At the first sight,
the right hand side of \eqref{S6} is unusual. Nonetheless, it arises naturally once we estimate the term $(\mathscr L_\mu V_{\vv,\beta,y})(x)$, where $\mathscr L_\mu$ is the infinitesimal generator of $(Y_t^\mu)_{t\ge0}$.

\item[{\rm(2)}] Seemingly, \eqref{S6} is a bit involved.  By imitating \eqref{E12},
 we can formally assume that $b$ satisfies the following condition: for some constants $\lambda_1,\lambda_2,\theta_i>0,i=1,2,3,4,$
 \begin{align*}
 \<x-y,b(x,\mu)\>\le-\lambda_1|x-y|^{1+\theta_1} +\big|\mbox{lower-order term}\big|+\lambda_2|x-y|^{\theta_2}\mu(|\cdot-y|^{\theta_3})^{\theta_4}.
 \end{align*}
 In this case, \eqref{S6}  can be checked easily whenever $\lambda_1,\lambda_2,\theta_i>0,i=1,2,3,4,$ are chosen suitably.
 However, the ``lower-order term'' might also provide dissipativity.   As shown in Examples \ref{exm} and \ref{exm2},
the dissipativity provided by the  ``lower-order term'' cannot be   neglected. Otherwise, it will be essentially  difficult to examine the condition (ii) in Theorem \ref{thm3}.
\end{itemize}
\end{remark}

Below, we provide two illustrative examples, which are adapted respectively from \cite[Example 3.5]{Zhang} and \cite[Example 3.6]{Zhang},
 to explain  the multiplicity  of stationary distributions for McKean-Vlasov SDEs with L\'{e}vy noise.
\begin{example}\label{exm}
Consider the McKean-Vlasov SDE on $\R$:
\begin{align}\label{exm-}
\d X_t=-\lambda X_t(X_t-a_1)(X_t-a_2)\,\d t-\kk \int_\R(X_t-y)\mathscr L_{X_t}(\d y)\,\d t+ \d Z_t,
\end{align}
where $a_1,a_2\in\R$ with $a_1a_2<0$, $\lambda,\kk>0$, and $(Z_t)_{t\ge0}$ is a $1$-dimensional symmetric $\alpha$-stable process with $\alpha\in(1,2)$. Assume that there exist
$(\kk,\lambda,\vv)\in(0,\8)^3$, $\beta\in(1,\alpha)$ and  $r_0\in(0,(|a_1|\wedge|a_2|)/4)$ such that
\begin{align}\label{-WE}
\kk/\lambda\ge  1+\frac{2 (a_1^2 - a_1a_2+a_2^2)}{ (\beta-1)(2+\beta)},
\end{align}
and
\begin{equation}\label{WE2}
\begin{split}
&   \frac{1}{\lambda\beta}\Big( \kappa \beta r^\beta_0+2^{\frac{\beta}{2}}\nu(|\cdot|^{\frac{\beta}{2}}\I_{\{|\cdot|>1\}})r^{\frac{\beta}{2}}_0 +\vv^{\frac{1}{2}\beta}\beta ( \lambda (a_1^2\vee a_2^2 -a_1a_2) +\kk )\\
&\quad\quad+ \nu(|\cdot|^{\beta}\I_{\{|\cdot|>1\}})+ \frac{\beta }{2 }\vv^{\frac{1}{2}\beta-1}\nu(|\cdot|^2\I_{\{|\cdot|\le 1\}})\Big)\\
&\le   r^\beta_0 [(r_0-|a_1|\wedge|a_2|)(r_0-|a_1-a_2|)+(\kappa/\lambda -1)].
\end{split}
\end{equation}
Then,  \eqref{exm-} has at least three stationary distributions. Note that,  \eqref{WE2} is valid as long as the intensity of the L\'{e}vy measure $\nu$ associated with $(Z_t)_{t\ge0}$ is sufficiently small since   $\vv$ can be chosen to be  small enough. Additionally,  \eqref{-WE} shows that the larger $\kk$ (i.e., the intensity of interaction) is beneficial to the occurrence of multiple stationary distributions.
\end{example}

Furthermore, a multidimensional example is established to show the multiplicity of stationary distributions.

\begin{example}\label{exm2}
Consider the following McKean-Vlasov SDE on $\R^d $:
\begin{equation}\label{WQ1}
\begin{split}
\d X_t&=-\frac{1}{2}\lambda\big((X_t-y_1)|X_t-y_2|^2+(X_t-y_2)|X_t-y_1|^2\big)\,\d t\\
&\quad-\kk \int_{\R^d}(X_t-y)\mathscr L_{X_t}(\d y)\,\d t+ \d Z_t,
\end{split}
\end{equation}
where $\lambda,\kk>0,$  $y_1,y_2\in\R^d,$  and $(Z_t)_{t\ge0}$ is a $d$-dimensional symmetric $\alpha$-stable process with $\alpha\in(1,2)$. It is easy to see that the leading term  of the drift  enjoys   a symmetric structure.
Assume that there exist
$(\kk,\lambda,\vv)\in(0,\8)^3$, $\beta\in(1,\alpha)$ and $r_0\in(0, |y_1-y_2|/4)$ such that
\begin{align} \label{EQ1}
\kk/\lambda\ge \vv+ \frac{(\beta^2+\beta+16)|y_1-y_2|^2 }{16(\beta+2)(\beta-1) },
\end{align}
and
\begin{equation}\label{WQ2}
\begin{split}
 &\frac{1}{\lambda\beta}\Big(\lambda\beta \vv   r_0^\beta+ 2^{\frac{\beta}{2}}\nu(|\cdot|^{\frac{\beta}{2}}\I_{\{|\cdot|>1\}})r_0^{\frac{\beta}{2}}+ \lambda\beta\vv^{\frac{\beta}{2}}\Big(\frac{1}{2}|y_1-y_2|^2+\frac{\kk}{\lambda}\Big)\\
&\quad\quad+\nu(|\cdot|^{\beta}\I_{\{|\cdot|>1\}}) +\frac{1}{2}\beta\vv^{\frac{1}{2}(\beta-2)}
\nu(|\cdot|^2\I_{\{|\cdot|\le 1\}})\Big) \\
&\le   r_0^\beta(r_0-|y_1-y_2|)(r_0-|y_1-y_2|/2).
\end{split}
\end{equation}
Thus, \eqref{WQ1} admits at least two stationary distributions.
\end{example}

\subsubsection{{\bf Uniqueness of stationary distributions}}
In this subsection, as regards  the SDE \eqref{E1},  we address  the problem  on uniqueness of stationary distribution.
Denote by $(P_t^{\mu})_{t\ge0}$ the Markov semigroup of the process $(Y_t^\mu)_{t\ge0}$ solving  \eqref{EW5}. For any $U\ge1$ and  $\mu_1,\mu_2\in\mathscr P$ satisfying $\mu_1(U)+\mu_2(U)<\8$,
 let $\|\mu_1-\mu_2\|_U$ be the weighted total variation distance defined by
$$\|\mu_1-\mu_2\|_U=\sup_{|\psi|\le U}|\mu_1(\psi)-\mu_2(\psi)|.
$$  Given  $r>0$, we set $V_r(x):=(1+|x|^2)^{r/2}$.
For the uniqueness, we need to impose additionally  two assumptions below. Let $\beta\in (0,\beta_*]$ as in Assumption {\rm(${\bf A}_1 $)}.
 \begin{enumerate}\it
\item[{\rm(${\bf A}_5 $)}]
 there exist constants $\beta_0\in (0, \beta\wedge\beta^*)$ and $K_1 >0$ such that
for all $t>0$, $\mu_1,\mu_2\in\mathscr P_{\beta_0}$ and $\mu\in\mathscr P_{\beta^*}^{M^*}$,
\begin{align}\label{UT}
\|\mu_1P_t^\mu-\mu_2P_t^\mu\|_{U}\le C_{M^*}\e^{-c_{M^*}t}\|\mu_1-\mu_2\|_{U},
\end{align}
and that for $x\in \R^d$ and $\mu_1,\mu_2\in \mathscr P_{\beta^*}^{M^*}$,
\begin{align}\label{WT5}
|b(x,\mu_1)-b(x,\mu_2)|\le K_1 V_{\beta\wedge \beta^*-\beta_0}(x)\|\mu_1-\mu_2\|_U,
\end{align}
where $\mu_0P_t^\mu$ denotes the law of $Y_t^\mu$ with $\mathscr L_{Y_0^\mu}=\mu_0\in\mathscr P_{\beta_0}$,
$U(x):=V_{\beta_0}(x)$, and $C_{M^*}, c_{M^*}>0$ depend on $M^*$.

\item[{\rm(${\bf A}_6 $)}] for some $\theta\in(0,1)$ and  any $p>1$, there is a locally integrable function $\varphi_p:=\varphi_{p,\theta, M^*}:[0,\8)\to[0,\8)$
 such that for  any $t>0$, $x\in \R^d$, $\mu\in\mathscr P_{\beta^*}^{M^*}$ and  $f\in C_b(\R^d)$,
\begin{align*}
|\nn P_t^{\mu}f|(x)\le \varphi_p(t) (1\wedge t)^{-\theta}  (P_t^{\mu}|f|^p)^{\frac{1}{p}}(x).
\end{align*}
\end{enumerate}

 The following theorem demonstrates that \eqref{E1} has a unique stationary distribution (which is also called an equilibria in literature), and reveals the convergence to the equilibria
  as long as the interaction intensity (i.e.,  $K_1 $  in  \eqref{WT5}) is tiny.

 \begin{theorem}\label{thm4}
Assume that Assumptions $({\bf A}_1)$, $({\bf A}_2)$,
$({\bf A}_3)$, $({\bf A}_5)$ and $({\bf A}_6)$  hold.
Then, there exists a constant $K^*>0$ such that,  if $K_1$ in \eqref{WT5} satisfies $K_1\in(0,K^*]$,  \eqref{E1} has a unique stationary distribution $\pi\in \mathscr P_{\beta^*}$;
and moreover there exist constants $C,\lambda>0$ such that for any $t\ge0$ and
$\mu\in\mathscr P_{\beta}$,
\begin{align}\label{WY3}
\|P_t^*\mu-\pi\|_U\le C\e^{-\lambda t}\|\mu-\pi\|_U,
\end{align}
where $P_t^*\mu:=\mathscr L_{X_t}$  with
 the initial distribution $\mathscr L_{X_0}=\mu.$
\end{theorem}

 \begin{remark}\label{re1.8}
 Let's interpret the aforementioned Assumptions $({\bf A}_5)$ and $({\bf A}_6)$.
 \begin{itemize}
 \item[{\rm (i)}] \eqref{UT} in Assumption (${\bf A}_5 $) shows the explicit contraction rate for the semigroup $(P_t^\mu)_{t\ge0}$ under the weighted total variation distance, which, along with Assumption (${\bf A}_1 $) (in particular, \eqref{e:add} holds), implies the exponential ergodicity of $(Y_t^\mu)_{t\ge0}$. Clearly, Assumption  $({\bf A}_6)$ implies the $C_b$-Feller property of the semigroup $(P_t^\mu)_{t\ge0}$. Hence, under Assumptions  $({\bf A}_1)$,  $({\bf A}_5)$ and $({\bf A}_6)$,   Assumption  (${\bf A}_2 $) is satisfied.

 \item[{\rm(ii)}]
 \eqref{UT} is to be examined in the Appendix section
 by imposing explicit conditions on the drift term $b(x,\mu)$ and the L\'{e}vy measure $\nu$; see in particular Proposition \ref{pro} for more details.   \eqref{WT5} indicates $\mu\mapsto b(\cdot,\mu)$ is continuous under the weighted total variation distance $\|\cdot\|_U$.
 Besides, one may want to adopt  the  following more general form instead of \eqref{UT}:
 \begin{align}\label{UT1}
\|\mu_1P_t^\mu-\mu_2P_t^\mu\|_{U}\le h_{M^*}(t)\|\mu_1-\mu_2\|_{U},
\end{align}
where $h_{M^*}$ is a non-negative decreasing function
on
 $[0,\infty)$ satisfying $\lim_{t\to \infty} h_{M^*}(t)=0$. However, according to the semigroup property of $(P_t^\mu)_{t\ge0}$ and the fact that $\mu_1,\mu_2\in\mathscr P_{\beta_0}$, we can easily see that \eqref{UT1}  is essentially the same as \eqref{UT}.
 \item[{\rm(iii)}]
  Assumption $({\bf A}_6)$ is the so-called gradient estimate for SDEs with jumps. So far, there are various methods (e.g., F.-Y. Wang's Harnack inequality \cite{WW} and the Bismut-Elworthy-Li formula \cite{Zhangxicheng}) to verify Assumption $({\bf A}_6)$. For the case that $(Z_t)_{t\ge0}$ is a symmetric $\alpha$-stable process,  Assumption $({\bf A}_6)$ will be validated in Corollary \ref{cor} below.\end{itemize}
 \end{remark}

\begin{corollary}\label{cor}
 Let $(Z_t)_{t\ge0}$
be a symmetric $\alpha$-stable process on $\R^d$ with $\alpha\in (1,2)$.
 Assume that Assumptions $({\bf A}_1)$, $({\bf A}_3)$  and $({\bf A}_5)$ hold, and
  that there exists a constant $K_2>0$ such that for any $x,y\in\R^d$ and $\mu\in\mathscr P_{\beta^*}$,
 \begin{align}\label{WT}
\<x-y,b(x,\mu)-b(y,\mu)\>\le K_2 |x-y|^2.
\end{align}
 Then, there exists a constant $K^*>0$
 such that, for $K_1$ in \eqref{WT5} satisfying
  $K_1 \in(0,K^*]$, \eqref{E1} has a unique stationary distribution, and meanwhile the statement \eqref{WY3} holds true.
 \end{corollary}

 Below, we provide an  example  to illustrate the application of Corollary  \ref{cor}.

\begin{example}\label{exa3} Let $(Z_t)_{t\ge0}$ be a $1$-dimensional symmetric $\alpha$-stable process with $\alpha\in (1,2)$.
Consider the following scalar McKean-Vlasov SDE with an asymmetric double-well  potential:
\begin{equation}\label{ETY}
\begin{split}
\d X_t=-\bigg( &\lambda X_t(X_t-1)(X_t+2)\,\d t\\
&-\kappa\int_\R\big((1+|X_t|^2)^{\frac{1}{2}(\beta-1)}|y|+g(y)\big)\mathscr L_{X_t}(\d y)\bigg)\,\d t +\d Z_t,
\end{split}
\end{equation}
where $\lambda,\kappa>0$, $\beta\in[1,\alpha)$ and $g:\R\to\R$ is a bounded continuous function.  Then, the SDE \eqref{ETY} admits a unique stationary distribution, and the convergence to the equilibria  (i.e., \eqref{WY3}) holds true provided that $\kappa>0$ is small enough.
\end{example}

For the sake of completeness,  in the Appendix section we shall  give another collection of sufficiency to check the uniqueness of stationary distributions for McKean-Vlasov SDEs with jumps, where the drift term involved is continuous under the  Wasserstein distance.
Note that, for Example \ref{exa3}, the drift term in \eqref{ETY} is no longer continuous under the Wasserstein distance since the function $g$ is merely bounded and continuous.

Before the end of this section, let us make some comments on the novelties of all main theorems stated  previously  and their approaches.
\begin{itemize}
\item[{\rm(1)}] In Theorem \ref{thm1} and Theorem \ref{thm3}, we present sufficient conditions to ensure the existence of stationary distributions  and  the  multiplicity of stationary distributions for the McKean-Vlasov SDE with jumps \eqref{E1}, respectively. The approaches of these two theorems are based on Schauder's fixed point theorem, which is essentially inspired by \cite{Tugaut, Zhang} for McKean-Vlasov SDEs driven by Brownian motion. However, it is not trivial to apply
     the ideas adopted in \cite{Tugaut, Zhang}
      to the SDE \eqref{E1}. For example, in the proof of Lemma \ref{lem1}, which is a key ingredient to that of Theorem \ref{thm1}, we need to split suitably  the jumping sizes of the associated L\'evy measure $\nu$; otherwise, one cannot derive the desired assertion under sharp conditions (i) and (ii) in Assumption $({\bf A}_1)$; see Remark \ref{Rem1.2}(2) for more details.   Furthermore, in the proof of Theorem \ref{thm3} we need to use the Lyapunov function $V_{\vv,\beta,y}(x):=(\vv+|x-y|^2)^{\frac{1}{2}\beta}$ with $\varepsilon>0$ small enough (instead of $V_{\beta,y}(x):=|x-y|^\beta$ which is not
       a
      $C^2$-function). Both two new points arise from the feature  of the McKean-Vlasov SDE with jumps and the fact that in the present work we only assume that the associated L\'evy measure $\nu$ only has finite $\beta_*\in(0,2]$-moment for large jumps.

     \item[{\rm(2)}] Banach's contraction mapping theorem is frequently used to prove the uniqueness of the stationary distribution  for McKean-Vlasov SDEs;
     see e.g. \cite{LMW,Wang, Wangb}. However, most of the known results are concerned with the contraction in term of the Wasserstein-type distance. As mentioned in \cite{BRS}, the convergence in the total variation for McKean-Vlasov SDEs is much more involved  than that in the Wasserstein-type distance. To the best of our knowledge, Theorem \ref{thm4} is the first one to establish the uniqueness of the stationary distribution  for McKean-Vlasov SDEs  
     driven by L\'evy noise, together  
     with the exponential convergence rates in terms of the total variation distance. In particular, in the proof of Theorem \ref{thm4}, we make full use of the gradient estimates for SDEs with jumps, which is interesting in its own and   also completely different from that in the framework for McKean-Vlasov SDEs with Brownian motion; see in particular \cite[Lemma 3.6]{BRS}.
\end{itemize}
 \ \

The content of this paper is organized as follows. In Section \ref{sec2},
 based on the  establishment that the mapping $\Lambda_\mu=\pi_\mu$ leaves invariantly in a suitable subspace of $\mathscr P_{\beta^*}$, we finish the proof of Theorem \ref{thm1} via Schauder's fixed point theorem. At the same time, some sufficient criteria are provided to examine Assumption (${\bf A}_3$) in Section \ref{sec2}. Section \ref{sec3} is devoted to the proofs of Theorem \ref{thm3}, Example \ref{exm} as well as Example \ref{exm2}. In Section \ref{sec4}, by the aid of gradient estimates for SDEs with jumps and Duhamel's principle, the proof of Theorem \ref{thm4} is completed.  Simultaneously, the proofs of Corollary \ref{cor} and Example \ref{exa3} are finished in  Section \ref{sec4}. In the last section, for completeness, we explore the uniqueness  of stationary distributions for McKean-Vlasov SDEs with jumps when the underlying drift is $L^1$-Wasserstein continuous with respect to the measure variable, and furnish  some sufficient conditions to check the exponential ergodicity \eqref{UT} in Assumption $({\bf A}_5)$ under the weighted total variation distance for the reference SDE \eqref{EW5}.

\section{Proof of Theorem \ref{thm1} and verification of Assumption (${\bf A}_3$)}\label{sec2}
The  reference  SDE \eqref{EW5}
provides a vital link  to prove  existence of stationary distributions for the SDE \eqref{E1}. Before    the proof of Theorem \ref{thm1}, we analyze in depth the behavior of the IPM
$\pi^\mu$ associated with  the SDE \eqref{EW5}.

\subsection{Qualitative analysis of $\pi^\mu$}
To begin,  we introduce some notations. Let $\beta\in (0,\beta_*]$, and $C_b, \theta_1,\lambda_1,\theta_2,\lambda_2$ be the constants in \eqref{E12}. Set $\beta^*:=\beta+\theta_1-1$.
Define for any $\gamma,\vv\in[0,1)$ and $a>0,$
\begin{align}\label{E14}
\Gamma(\gamma,a,\vv):=(1-\vv)\big(\I_{\{\gamma=0\}}+ a^{ \frac{\vv}{1-\vv}}\I_{\{\gamma\in(0,1)\}}\big),
\end{align}
and  for any $\vv, r >0 $ and $l>1$,
\begin{equation}\label{E3}
\begin{split}
\Phi(\vv, r,l)
&:=  \beta  C_b +\beta\Gamma((\beta-1)_+, \gamma_2/\vv,\gamma_2)\left|\nu(\cdot\I_{\{1<|\cdot|\le l\}})\right| ^{\frac{1}{1-\gamma_2}}\\
&\quad\,\,+\beta \lambda_1h(r^2)^{(1+\theta_1)/2} (1+r^2)^{\beta^*/2}
 +  \frac{1}{2}\beta\nu(|\cdot|^2\I_{\{|\cdot|\le l\}})+\nu(|\cdot|^{\beta}\I_{\{|\cdot|>l\}}),
\end{split}
\end{equation}
 where
  $\gamma_2: =(\beta-1)^+/\beta^*$ and $h(r):=\frac{r}{1+r}.$ Additionally, we define the following  index set:
\begin{equation}\label{E0}
\begin{split}
\Theta =\big\{(\vv_1,\vv_2, r_0, l)\in(0,\8)^3\times[1,\8):\beta&\big(\lambda_1h(r_0^2)^{(1+\theta_1)/{2}}-\vv_1\lambda_2\I_{\{\gamma_1\in(0,1)\}}-\vv_2\big)\\
&- 2^{\frac{\beta}{2}}\nu(|\cdot|^{\frac{\beta}{2}}\I_{\{|\cdot|>l\}})>0 \big\},
\end{split}
\end{equation} where $\gamma_1:=(\beta+\theta_2-2)^+/\beta^*$.

The following lemma demonstrates  that, for fixed $\mu\in\mathscr P_{\beta^*}$,  the $\beta^*$-th moment of $\pi^\mu$ can be bounded explicitly.

\begin{lemma}\label{lem1}
Assume that \eqref{E12} holds  with $\theta_1\ge1-\beta/2$ and $\theta_2<1+\theta_1$, and that, for fixed $\mu\in \mathscr P_{\beta^*}$, the process $(Y_t^{\mu})_{t\ge0}$ associated with \eqref{EW5} satisfies Assumption $({\bf A}_2)$ and is non-explosive.
Then, for any $(\vv_1,\vv_2,r_0,l)\in \Theta$,
\begin{align}\label{E-}
\pi^\mu(|\cdot|^{\beta^*} )\le  \frac{  \Phi(\vv_2, r_0,l)+\beta\lambda_2\Gamma(\gamma_1,\gamma_1/\vv_1,\gamma_1)\mu(|\cdot|^{\theta_3})^{\frac{\theta_4}{1-\gamma_1}} }{\beta(\lambda_1h(r_0^2)^{\frac{1}{2}(1+\theta_1)}-\vv_1\lambda_2\I_{\{\gamma_1\in(0,1)\}}-\vv_2)-2^{\frac{\beta}{2}}
\nu(|\cdot|^{\frac{\beta}{2}}\I_{\{|\cdot|>l\}})}.
\end{align}
In particular, this implies $\pi^\mu\in \mathscr P_{\beta^*}$.
\end{lemma}

\begin{proof}

Below,   we  stipulate $(\vv_1,\vv_2, r_0,l)\in\Theta$,
fix $\mu\in\mathscr P_{\beta^*}$, and recall $V_r(x)=(1+|x|^2)^{\frac{r}{2}}$ for all $x\in\R^d$ and a given  $r\ge0.$
In some occasions, we shall
write $(Y_t^{\mu,x})_{t\ge0}$ instead of $(Y_t^\mu)_{t\ge0}$ to stress  the initial value $Y_0^\mu=x\in\R^d$.
 Furthermore, we
 define the empirical measure:
 $$\Lambda_t^\mu(B) =\frac{1}{t}\int_0^t(\delta_{\bf0}P_s^\mu)(B)\,\d s,\quad t>0,\, B\in\mathscr B(\R^d),$$
 where $(\delta_xP_t^\mu)_{t\ge0}$ denotes the
 distribution of $(Y_t^{\mu,x})_{t\ge0}$.
By  Chebyshev's  inequality, it follows  that for any $t,R>0,$
\begin{align*}
\Lambda_t^\mu\big(\I_{\{V_{\beta^*}(\cdot)\ge R\}}\big)\le\frac{1}{tR}\int_0^t\E V_{\beta^*}(Y_s^{\mu,{\bf0}})\,\d s.
\end{align*}
Therefore, due  to the fact that
$\R^d\ni x\mapsto V_{\beta^*}(x)$ is a compact function,
$(\Lambda_t^\mu)_{t\ge1}$ is tight provided that the following statement
\begin{equation}\label{E4}
\begin{split}
\int_0^t\E V_{\beta^*}(Y_s^{\mu,{\bf0}})\,\d s\le &\frac{1+\big( \Phi(\vv_2,r_0, l)+\beta\lambda_2\Gamma(\gamma_1,\gamma_1/\vv_1,\gamma_1)\mu(|\cdot|^{\theta_3})^{\frac{\theta_4}{1-\gamma_1}}\big)t}
{\beta(\lambda_1h(r_0^2)^{\frac{1+\theta_1}{2}}-\vv_1\lambda_2\I_{\{\gamma_1\in(0,1)\}}-\vv_2)-2^{\frac{\beta}{2}}
\nu(|\cdot|^{\frac{\beta}{2}}\I_{\{|\cdot|>l\}})}
\end{split}
\end{equation}
is true.
Subsequently,  by invoking the Krylov-Bogoliubov criterion (see e.g. \cite[Theorem 3.1.1]{DZ}) and making use of  the $C_b$-Feller property of the process $(Y_t^{\mu})_{t\ge0}$ as well as the uniqueness of $\pi^\mu$,
we deduce   that for any $n\ge1$,
\begin{align*}
\pi^\mu(V_{\beta^*}(\cdot)\wedge n)=\limsup_{t\to\8}\Lambda^\mu_t(V_{\beta^*}(\cdot)\wedge n)\le \limsup_{t\to\8}(n\wedge\Lambda^\mu_t(V_{\beta^*}(\cdot))),
\end{align*}
where the inequality is valid due to Jensen's inequality plus the concave property of the mapping  $[0,\8)\ni r\mapsto r\wedge n$.
Accordingly,
the assertion \eqref{E-} is achieved   via Fatou's lemma and the prerequisite based on the establishment of \eqref{E4}.

In the following context, we aim at  verifying \eqref{E4}. Applying It\^o's formula  yields that
\begin{equation*}
\begin{split}
\d V_\beta(Y_t^\mu) &=  \beta V_{\beta-2}(Y_t^\mu)\<Y_t^\mu,b(Y_t^\mu,\mu)\>\,\d t \\
&\quad +\int_{\{|z|\le 1\}}\big\{V_\beta(Y_t^\mu+z)-V_\beta(Y_t^\mu)-\<\nn V_\beta(Y_t^\mu),z\> \big\}\,\nu(\d z)\,\d t \\
&\quad  +\int_{\{|z|>1\}}\big\{V_\beta(Y_t^\mu+z)-V_\beta(Y_t^\mu) \big\}\,\nu(\d z)\, \d t+\d M_t^\mu,
\end{split}
\end{equation*} where $(M_t^{\mu})_{t\ge0}$ is a local martingale.
It is obvious that the identity above can be reformulated as follows: for all $l\ge 1$,
\begin{equation}\label{E6}
\begin{split}
\d V_\beta(Y_t^\mu) &=  \beta V_{\beta-2}(Y_t^\mu)\left(\<Y_t^\mu,b(Y_t^\mu,\mu)\>+\int_{\{1\le |z|\le l\}}\<Y_t^\mu,z\>\nu(\d z)\right)\,\d t \\
&\quad +\int_{\{|z|\le l\}}\big\{V_\beta(Y_t^\mu+z)-V_\beta(Y_t^\mu)-\<\nn V_\beta(Y_t^\mu),z\> \big\}\,\nu(\d z)\,\d t \\
&\quad  +\int_{\{|z|>l\}}\big\{V_\beta(Y_t^\mu+z)-V_\beta(Y_t^\mu) \big\}\,\nu(\d z)\, \d t+\d M_t^\mu\\
&=:  \big( I_1^\mu(t)+I_2^\mu(t)+I_3^\mu(t)\big) \,\d t+ \d M_t^{\mu}.
\end{split}
\end{equation}

From   $\beta\in (0,\beta_*]\subset (0,2]$, besides   $\nn V_\beta(x)=\beta V_{\beta-2}(x)x,x\in\R^d, $ and \eqref{E12}, we deduce that
\begin{equation}\label{E5}
\begin{split}
I_1^\mu(t)&\le \beta V_{\beta-2}(Y_t^\mu)\big(C_b-\lambda_1|Y_t^\mu|^{1+\theta_1} +\lambda_2 (1+|Y_t^\mu|^2)^{\theta_2/2}\mu(|\cdot|^{\theta_3})^{\theta_4}
\big)\\
&\quad +\beta V_{\beta-2}(Y_t^\mu) |Y_t^\mu|\left|\nu(\cdot\I_{\{1<|\cdot|\le l\}})\right|\\
&\le \beta  C_b -\beta\lambda_1h(|Y_t^\mu|^2)^{\frac{1+\theta_1}{2}}V_{\beta^* }(Y_t^\mu) +\beta\lambda_2V_{(\beta+\theta_2-2)^+}(Y_t^\mu) \mu(|\cdot|^{\theta_3})^{\theta_4} \\
&\quad+\beta V_{\beta-1}(Y_t^\mu)\left|\nu(\cdot\I_{\{1<|\cdot|\le l\}})\right|,
\end{split}
\end{equation}
where $h(r):=r/(1+r), r\ge0$.
Next,
Young's inequality implies that
\begin{align*}
 V_{(\beta+\theta_2-2)^+}(x) \mu(|\cdot|^{\theta_3})^{\theta_4}&\le \vv_1V_{\beta^* }(Y_t^\mu) \I_{\{\gamma_1\in(0,1)\}}  +
   \Gamma(\gamma_1,\gamma_1/\vv_1,\gamma_1)\mu(|\cdot|^{\theta_3})^{\frac{\theta_4}{1-\gamma_1}} \end{align*}
 and that
\begin{align*} V_{\beta-1}(Y_t^\mu)\left|\nu(\cdot\I_{\{1<|\cdot|\le l\}})\right|&\le \varepsilon_2 V_{\beta^*}(Y_t^\mu)\I_{\{\beta\in(1,2)\}} +\Gamma((\beta-1)_+,\gamma_2/\vv_2,\gamma_2)\left|\nu(\cdot\I_{\{1<|\cdot|\le l\}})\right| ^{\frac{1}{1-\gamma_2}},
\end{align*}
 where $\gamma_1:=(\beta+\theta_2-2)^+/\beta^*$, $\gamma_2:=(\beta-1)_+/\beta^*$, and $\Gamma $ was defined in \eqref{E14}.
 Correspondingly,  $I_1^\mu(t)$ can be bounded as below:
\begin{equation}\label{E7}\begin{split}
I_1^\mu(t)\le  &\beta  C_b+\beta \lambda_1h(r_0^2)^{\frac{1+\theta_1}{2}}(1+r_0^2)^{\frac{1}{2}\beta^*}
\\
&- \beta\big(\lambda_1h(r_0^2)^{\frac{1+\theta_1}{2}}-\vv_1\lambda_2\I_{\{\gamma_1\in(0,1)\}} -\vv_2\I_{\{\beta\in(1,2)\}}\big)V_{\beta^* }(Y_t^\mu)\\
&+\beta\lambda_2\Gamma(\gamma_1,\gamma_1/\vv_1,\gamma_1)\mu(|\cdot|^{\theta_3})^{\frac{\theta_4}{1-\gamma_1}}+\beta\Gamma((\beta-1)_+,\gamma_2/\vv_2,\gamma_2)\left|\nu(\cdot\I_{\{1<|\cdot|\le l\}})\right| ^{\frac{1}{1-\gamma_2}},
\end{split}
\end{equation}
where we also employed the fact that $(0,\8)\ni r\mapsto h(r)$  is increasing.

We move to estimate the term $I^\mu_2(t).$
A direct calculation shows that
\begin{align*}
\nn^2 V_\beta(x)=\beta                                                                                                                                                                                V_{\beta-2}(x)I_d+\beta(\beta-4)V_{\beta-4}(x)xx^\top,\quad x\in\R^d,
\end{align*}
where $x^\top$ means the transpose of $x,$  and $I_d$ stands for the identity $d\times d$-matrix. Then, the mean value theorem and the precondition $\beta\in(0,2)$
 enable  us to obtain that
\begin{align}\label{E8}
 I_2^\mu(t)\le  \frac{\beta }{2}\nu(|\cdot|^2\I_{\{|\cdot|\le l\}}).
\end{align}

Next,
we turn to quantify  $I_3^\mu(t).$
By invoking   the  inequality: $|a^{\beta/2}-b^{\beta/2}|\le |a-b|^{\beta/2}$ for all $a,b\ge0$,
it is easy to see that for all $x\in\R^d,$
\begin{align*}
\int_{\{|z|>l\}}|V_\beta(x+z)-V_\beta(x)|\,\nu(\d z)
&\le \int_{\{|z|>l\}}\big(2|x|\cdot|z|+|z|^2\big)^{\frac{\beta}{2}}\,\nu(\d z)\\
&\le 2^{\frac{\beta}{2}}\nu(|\cdot|^{\frac{\beta}{2}}\I_{\{|\cdot|>l\}})|x|^{\frac{\beta}{2}}+\nu(|\cdot|^{\beta}\I_{\{|\cdot|>l\}})\\
&\le 2^{\frac{\beta}{2}}\nu(|\cdot|^{\frac{\beta}{2}}\I_{\{|\cdot|>l\}}) V_{\beta^*}(x)+\nu(|\cdot|^{\beta}\I_{\{|\cdot|>l\}}),
\end{align*} where in the last inequality we used the prerequisite $\theta_1\ge1-\beta/2$ (so $\beta^*\ge \beta/2$).
So,
 $I^\mu_3(t)$  can be dominated as below:
\begin{equation}\label{E11}
\begin{split}
 I_3^\mu(t)&\le 2^{\frac{\beta}{2}}\nu(|\cdot|^{\frac{\beta}{2}}\I_{\{|\cdot|>l\}})
 V_{\beta^* }(Y_t^\mu)  +\nu\big(|\cdot|^{\beta}\I_{\{|\cdot|>l\}}\big).
\end{split}
 \end{equation}
 Now, putting \eqref{E7}, \eqref{E8} and \eqref{E11} into \eqref{E6}, we arrive at
\begin{equation}\label{E10}
\begin{split}
\d  V_\beta(Y_t^\mu)
\le&
 - \big(\beta\big(\lambda_1h(r_0^2)^{\frac{1+\theta_1}{2}}-\vv_1\lambda_2\I_{\{\gamma_1\in(0,1)\}}-\vv_2\big)-2^{\frac{\beta}{2}}\nu(|\cdot|^{\frac{\beta}{2}}\I_{\{|\cdot|>l\}}) \big)V_{\beta^* }(Y_t^\mu)\,\d t\\
 &+  \beta\lambda_2\Gamma(\gamma_1,\gamma_1/\vv_1,\gamma_1)\mu(|\cdot|^{\theta_3})^{\frac{\theta_4}{1-\gamma_1}}\d t +\Phi(\vv_2, r_0,l)
 \,\d t+\d M_t^\mu,
\end{split}
\end{equation}
where $\Phi$ was  defined in \eqref{E3}.

For any $m>0$, define the stopping time:
\begin{align*}
\tau_m^\mu=\big\{t>0: |Y_t^\mu|>m\big\}.
\end{align*}
Since the process $(Y_t^\mu)_{t\ge0}$ is non-explosive, $\lim_{m\to\8}\tau_m^\mu=\8$, a.s. As a consequence,
Fatou's lemma   yields  that
\begin{align*}
 \E\bigg(\int_0^t V_{\beta^* }(Y_s^{\mu,{\bf0}})\,\d s\bigg) =\E\bigg(\liminf_{m\to\8} \int_0^{t\wedge \tau_m^\mu} V_{\beta^* }(Y_s^{\mu,{\bf0}})\,\d s\bigg) \le\liminf_{m\to\8}\E\bigg( \int_0^{t\wedge \tau_m^\mu} V_{\beta^* }(Y_s^{\mu,{\bf0}})\,\d s\bigg).
\end{align*}
Subsequently,
\eqref{E4} is verifiable by taking  \eqref{E10} into consideration.
\end{proof}

According to Lemma \ref{lem1}, under Assumptions $({\bf A}_1)$ and $({\bf A}_2)$,
 the mapping
\begin{align}\label{S10}
\Lambda_\cdot: \mathscr P_{\beta^*}\to \mathscr P_{\beta^*},\quad\quad  \mu\mapsto\Lambda_\mu=\pi^\mu
\end{align}
is well defined. The lemma below shows that, for a suitable $M^*>0,$
 $ \mathscr P_{\beta^*}^{M^*}$ is an invariant set of the mapping $\Lambda_\cdot$.

\begin{lemma}\label{lem2}
Under Assumptions $({\bf A}_1)$
and $({\bf A}_2)$, $\Lambda_\mu\in\mathscr P_{\beta^*}^{M} $ for given $\mu\in\mathscr P_{\beta^*}^{M^*}$, where $M^*:=M_1^*\vee M_2^*$ with $M_1^*,M_2^*>0$ being given in \eqref{T1} and \eqref{T2} below, respectively.
\end{lemma}

\begin{proof} Firstly, we note that, for fixed $\mu\in \mathscr P_{\beta^*}$, the process $(Y_t^{\mu})_{t\ge0}$ associated with \eqref{EW5} is non-explosive under Assumption $({\bf A}_1)$; see the argument of \eqref{EW3} below for more details. Hence, Lemma \ref{lem1} holds true.
In the following analysis,  we intend to show the desired assertion, case by case.

(i) Concerning the case  (${\bf A}_1) (i)$, we choose $
l>1$ large enough so that $$2^{\frac{\beta}{2}}\nu(|\cdot|^{\frac{\beta}{2}}\I_{\{|\cdot|>l\}})\le  2^{-\frac{1}{2}(7+\theta_1)}\beta\lambda_1.$$
In addition, we take
\begin{align}\label{T1}
\vv_1^{*}   =\frac{\lambda_1}{ 2^{ \frac{1}{2}(3+\theta_1)}\lambda_2}\quad \mbox{ and } \quad M_1^*  =\frac{2^{ 2+\frac{1}{2}(5+\theta_1)}(\Phi(\vv_2^{*},1, l)+ C_1^*)}{3\beta\lambda_1},
\end{align}
where for $\gamma_3=\frac{\theta_3\theta_4}{\beta^*(1-\gamma_1)}\in(0,1),$
\begin{align*} \vv_2^{*}: =\frac{\lambda_1 }{2^{ \frac{1}{2}(
7+\theta_1)}}~ \mbox{ and }  ~
C_1^*  =(1-\gamma_3 )\big(  \beta\lambda_2\Gamma(\gamma_1,\gamma_1/{\vv_1^*},\gamma_1) \big)^{\frac{1}{1-\gamma_3}}\Big(\frac{2^{ \frac{1}{2}(7+\theta_1)}\gamma_3}{\beta\lambda_1}\Big)^{\frac{\gamma_3}{1-\gamma_3}}.
\end{align*}
It is easy to see that
\begin{align*}
\beta\big(\lambda_1h(1)^{\frac{1}{2}(1+\theta_1)}-\vv_1^{*}\lambda_2\I_{\{\gamma_1\in(0,1)\}}-\vv_2^{*}\big)-2^{\frac{\beta}{2}}\nu(|\cdot|^{\frac{\beta}{2}}\I_{\{|\cdot|>l\}})
\ge2^{ -\frac{1}{2}(5+\theta_1)}\beta\lambda_1.
\end{align*}
 Hence,
$(\vv_1^{*},\vv_2^{*},1,l)\in\Theta$. Next, by H\"older's inequality and Young's inequality, it follows that
\begin{align*}
 \beta\lambda_2\Gamma(\gamma_1,\gamma_1/{\vv_1^*},\gamma_1)\mu(|\cdot|^{\theta_3})^{\frac{\theta_4}{1-\gamma_1}}
  \le  2^{ -\frac{1}{2}(7+\theta_1)}\beta\lambda_1\mu(|\cdot|^{\beta^*})+C_1^*.
\end{align*}
So, for any $\mu\in  \mathscr P_{\beta^*}^{M^*_1} $,  we can deduce from \eqref{E-} that
\begin{align*}
\pi^\mu(|\cdot|^{\beta^*} )\le \frac{  2^{ \frac{1}{2}(5+\theta_1)}(\Phi(\vv_2^{*},1,l)+2^{- \frac{1}{2}(7+\theta_1)}\beta\lambda_1M_1^*+C_1^*)}{\beta\lambda_1
}.
\end{align*}
This  yields that, for any  $\mu\in  \mathscr P_{\beta^*}^{M^*_1}$,  $\Lambda_\mu\in  \mathscr P_{\beta^*}^{M^*_1}$
 by leveraging  the definition of $M_1^*$.

(ii) Regarding the setting  (${\bf A}_1) (ii)$, we take $l>1$ large enough so that
$$2^{\frac{\beta}{2}}\nu(|\cdot|^{\frac{\beta}{2}}\I_{\{|\cdot|>l\}})\le\frac{1}{8}\beta( \lambda_1 -\lambda_2).$$
Moreover, we stipulate
\begin{align}\label{T2}
\vv_1^{**} =\gamma_1\quad \mbox{ and } \quad  M_2^*=\frac{4\Phi(\vv_2^{**},r_0, l)}{\beta(\lambda_1 -\lambda_2 ) },
\end{align}
where for $\lambda^*:=\frac{1}{2}(\lambda_1+\lambda_2)\in(\lambda_2,\lambda_1)$,
\begin{align*}
\vv_2^{**}:=\frac{1}{8 }\big( \lambda_1-\lambda_2 \big)\quad \mbox{ and } \quad
r_0:=
\frac{ (\lambda^*/\lambda_1)^{\frac{1}{1+\theta_1}}  }{\big(1-
 (\lambda^*/\lambda_1)^{\frac{2}{1+\theta_1}}  \big)^{\frac{1}{2}}}.
\end{align*}
It is ready  to see that $h(r_0^2)^{\frac{1}{2}(1+\theta_1)}=\lambda^*/\lambda_1$. This, together with the alternatives of $\vv_1^{**}$
and $\vv_2^{**}$, leads to
\begin{align*}
\beta \big(\lambda_1h(r_0^2)^{\frac{1}{2}(1+\theta_1)}-\vv_1^{**}\lambda_2\I_{\{\gamma_1\in(0,1)\}}
-\vv_2^{**}\big)-2^{\frac{\beta}{2}}\nu(|\cdot|^{\frac{\beta}{2}}\I_{\{|\cdot|>l\}})
  &\ge \frac{1}{2}\beta\lambda^*+\beta\big(1/2-\gamma_1\big)\lambda_2\\
  & >\beta\lambda_2(1-\gamma_1),
\end{align*}
where the second inequality is valid due to $\lambda_2<\lambda^*.$ As a result, we infer that
$(\vv_1^{**},\vv_2^{**},r_0,l)\in\Theta$. Next, with the aid of    $\theta_3\theta_4=\beta^*(1-\gamma_1)$ and $\Gamma(\gamma_1,1,\gamma_1)=1-\gamma_1$,
we obtain from $\vv_1^{**}=\gamma_1$ and H\"older's inequality that
$$
 \beta\lambda_2\Gamma(\gamma_1,1,\gamma_1)\mu(|\cdot|^{\theta_3})^{\frac{\theta_4}{1-\gamma_1}} \le \beta\lambda_2(1-\gamma_1)\mu(|\cdot|^{\beta^*}).
$$
Whereafter, we derive from \eqref{E-} that for any $\mu\in  \mathscr P_{\beta^*}^{M^*_2} $,
 \begin{align*}
\pi^\mu(|\cdot|^{\beta^*} )\le \frac{ \Phi( \vv_2^{**},r_0,l)+ \beta\lambda_2(1-\gamma_1)M_2^* }{  \beta\lambda^*/2+\beta\big(1/2-\gamma_1\big)\lambda_2}.
\end{align*}
Whence, we conclude that, for  $\mu\in\mathscr P_{\beta^*}^{M^*_2}$, $\pi^\mu(|\cdot|^{\beta^*} )\le M_2^*$  by
invoking the definition of $M_2^*.$
\end{proof}

\subsection{Proof of Theorem \ref{thm1}}
Inspired by that of  \cite[Theorem 2.2]{Zhang}, we  complete the proof of Theorem \ref{thm1} by means of  Schauder's fixed point theorem in a real Banach space.
Note that the space $ \mathscr P_{\beta^*}$ endowed with the Wasserstein distance  is not a Banach space so that Schauder's fixed point theorem cannot be applied
 directly. To this end, we need to enlarge suitably the space $ \mathscr P_{\beta^*}$ such that the corresponding enlarged space is a Banach space and the associated restriction coincides with a subspace of $ \mathscr P_{\beta^*}$. For this purpose, we introduce the following space: for given $p>0,$
\begin{align*}
\mathcal M_p:=\big\{\mu: \mu \mbox{ is a finite signed measure on } \R^d \mbox{ satisfying } |\mu|(|\cdot|^p)<\8\big\},
\end{align*}
where $|\mu|$ stands for the variation of the signed measure $\mu.$ Under the Kantorovich-Rubinstein
 metric:
\begin{align*}
\mathcal W_{\beta^*/2}(\mu_1,\mu_2):= \sup_{(f,g)\in \mathscr F_{\rm Lip,1\vee (\beta^*/2)}}|\mu_1(f)-\mu_2(g)|^{\frac{2}{2\vee\beta^*}},\quad \mu_1,\mu_2\in\mathcal M_{\beta^*/2},
\end{align*}
 $\mathcal M_{\beta^*/2}$ is a Banach space,  where
\begin{align*}
\mathscr F_{\rm Lip,1\vee (\beta^*/2)}:=\big\{(f,g):  &f,g \mbox{ are Lipschitz  on }   \R^d  \mbox{ satisfying } \\
&f(x)\le g(x)+|x-y|^{1\vee ({\beta^*}/{2})} \hbox{ for } x,y\in\R^d \big\}.
\end{align*}
Obviously,  $\mathcal W_{\beta^*/2}(\mu_1,\mu_2)$ coincides with  $\mathbb W_{\beta^*/2}(\mu_1,\mu_2)$ for $\mu_1,\mu_2\in \mathscr P_{\beta^*/2}$, in which, for $p>0,$ $\mathbb W_{p}$ represents the standard  $L^{p}$-Wasserstein distance defined by
\begin{align}\label{T6}
\mathbb W_{p}(\mu_1,\mu_2)=\inf_{\pi\in\mathscr C(\mu_1,\mu_2)}\bigg(\int_{\R^d\times\R^d}|x-y|^{p}\pi(\d x,\d y)\bigg)^{\frac{1}{1\vee p}}.
\end{align}
Herein, $\mathscr C(\mu_1,\mu_2)$ means the collection of couplings of $\mu_1,\mu_2.$ See \cite[Section 3]{Rachev} for more details.

With the preceding warm-up materials at hand, we are in position to  implement  the proof of Theorem \ref{thm1}.
\begin{proof}[Proof of Theorem $\ref{thm1}$]
Recall that,  for fixed $\mu\in\mathscr P_{\beta^*}$, $\pi^\mu$ is the unique IPM of $(Y_t^\mu)_{t\ge0}$; see Assumption (${\bf A}_2$) for details. Then, if $\mathscr L_{Y_0^\mu}=\pi^\mu$, we have $\mathscr L_{Y_t^\mu}=\pi^\mu$ for all $t\ge0$.
Hence, provided that we can claim that the mapping $\mu\mapsto \Lambda_\mu=\pi^\mu$ has a fixed point (still written as $\mu$)
in the subspace $\mathscr P_{\beta^*}^{M^*}$, it follows that $\mathscr L_{Y_t^\mu}=\pi^\mu=\mu=\mathscr L_{Y_0^\mu}$ for all $t\ge0.$
Consequently, we can conclude that $(X_t)_{t\ge0}$ solving \eqref{E1} possesses a stationary distribution by invoking the weak uniqueness of \eqref{E1}.

 On the basis of the aforementioned analysis,
it remains to verify  that the mapping $\mathscr P_{\beta^*}^{M^*} \ni\mu\mapsto \Lambda_\mu$ has a fixed point. According to Schauder's fixed point theorem (see e.g. \cite[Theorem 8.8]{Deimling}), it suffices to show respectively that (i) $\mathscr P_{\beta^*}^{M^*}$ is    convex  and compact in  $(\mathscr P_{ \beta^*/2}, \mathbb W_{ \beta^*/2})$; (ii) $ \mathscr P_{\beta^*}^{M^*} \ni \mu\mapsto \pi^\mu$ is weakly continuous, which indeed is imposed in Assumption $({\bf A}_3)$.
The convex property of  $\mathscr P_{\beta^*}^{M^*}$ is obvious. To show that $\mathscr P_{\beta^*}^{M^*}$ is a compact subset,
it is sufficient to verify (a)  $ \mathscr P_{\beta^*}^{M^*} $ is weakly compact; (b) $ \lim_{n\to\8}\sup_{\mu\in\mathscr P_{\beta^*}^{M^*}}\mu(|\cdot|^{\frac{1}{2}\beta^*}\I_{\{|\cdot|\ge n\}})=0;$ see, for instance, \cite[Theorem 5.5, p.175]{Chen}.
By Chebyshev's inequality, it is easy to see that $ \mathscr P_{\beta^*}^{M^*} $ is tight so that it is weakly compact.
Next, applying Chebyshev's inequality once more yields that  for any $\mu\in\mathscr P_{\beta^*}^{M^*}$,
\begin{align*}
\mu(|\cdot|^{\frac{1}{2}\beta^*}\I_{\{|\cdot|\ge n\}})\le n^{-\frac{1}{2}\beta^*}\mu(|\cdot|^{\beta^*})\le n^{-\frac{1}{2}\beta^*}M^*.
\end{align*}
Therefore,  (b) is verifiable.
\end{proof}

\subsection{Verification of Assumption $({\bf A}_3)$}
 In this subsection, we plan  to provide  some examinable conditions
  to check $({\bf A}_3)$. For this, we assume the following assumptions:

\begin{enumerate}\it
\item[{\rm$({\bf A}_{31})$}] for
$\mu\in \mathscr P_{\beta^*}^{M^*}$,
 there exists  a function $h_{\mu}:[0,\8)\to[0,\8)$ satisfying  $\lim_{t\to\8}h_{\mu}(t)=0$  such that
\begin{align*}
 \|\pi^\mu - \delta_x P_t^\mu  \|_{\rm{var}}\le  (1+|x|^{\beta^*})  h_{\mu}(t),
\end{align*}
where $M^*>0$ is given in Lemma $\ref{lem2}$, and
$\|\cdot\|_{\rm{var}}$ stands for the total variation norm.

\item[{\rm$({\bf A}_{32})$}] for  given   $n \ge1$ and  $\mu\in \mathscr P_{\beta^* }^{M^*}$,
there exists a   constant  $K_{n,\mu}>0 $ such that for all  $x,y\in\R^d$ with $|x|\vee|y|\le n,$
\begin{align}\label{E9-}
 \<x-y,b(x,\mu)-b(y,\mu)\>\le K_{n,\mu} |x-y|^2,
\end{align}
 and
\begin{align}\label{EW}
\lim_{\bar\mu\in\mathscr P_{\beta^*}^{M^*},\bar\mu\overset{w}{\rightarrow}\mu} \Big(\sup_{x\in\R^d,|x|\le n}|b(x,\bar\mu)-b(x,\mu)|\Big) =0,
\end{align}
where $\bar\mu\overset{w}{\rightarrow}\mu$ means that $\bar\mu$ converges weakly to $\mu.$
\end{enumerate}

Then, we have the following statement, which reveals that the mapping $\mathscr P_{\beta^*}^{M^*}  \ni\mu\mapsto \pi^\mu$  is weakly continuous.

 \begin{proposition}\label{pro1}
Assume that Assumptions $({\bf A}_1)$, $({\bf A}_2)$, $({\bf A}_{31})$ and $({\bf A}_{32})$  hold.
Then, for any $f\in C_b(\R^d)$ and $\mu\in\mathscr P_{\beta^*}^{M^*} $,
\begin{align}\label{E7-}
\lim_{\bar\mu\in\mathscr P_{\beta^*}^{M^*}, \bar\mu\overset{w}{\rightarrow}\mu}|\pi^{\bar\mu}(f)-\pi^{\mu}(f)|=0.
\end{align} In particular, Assumption $({\bf A}_3)$ is satisfied.
\end{proposition}

\begin{proof}
To achieve  \eqref{E7-},  it suffices to prove that for any   $f\in{\rm{Lip}}_b(\R^d)\cap C_b(\R^d)$ and  $\mu\in\mathscr P_{\beta^*}^{M^*} $,
\begin{align}\label{EW1}
\lim_{\bar\mu\in\mathscr P_{\beta^*}^{M^*}, \bar\mu\overset{w}{\rightarrow}\mu}|\pi^{\bar\mu}(f)-\pi^{\mu}(f)|=0.
\end{align}
Below, we stipulate  $\mu,\bar\mu\in\mathscr P_{\beta^*}^{M^*} $, and fix $f\in\rm{Lip}_b(\R^d)\cap C_b(\R^d)$.
Via the invariance of $\pi^{\mu}$ and $\pi^{\bar\mu}$,
it is ready to see that for any $t\ge0,$
\begin{equation}\label{WE}
\begin{split}
|\pi^{\mu}(f)-\pi^{\bar\mu}(f)|&\le| \pi^{\mu}(f)-(\pi^{\bar\mu}P_t^{\mu})(f)|+|(\pi^{\bar\mu}P_t^{\mu})(f)-(\pi^{\bar\mu}P_t^{\bar\mu})(f)|\\
&\le \|f\|_\8\|\pi^{\mu } -\pi^{\bar\mu}P_t^{\mu}\|_{\rm{var}}+ \pi^{\bar\mu}(|P_t^{\mu}(f)- P_t^{\bar\mu}(f)|)\\
&\le h_{\mu}(t)\|f\|_\8\big(1+\pi^{\bar\mu}(|\cdot|^{\beta^*})\big)+ \pi^{\bar\mu}(|P_t^{\mu}(f)- P_t^{\bar\mu}(f)|)\\
&\le h_{\mu}(t)\|f\|_\8\big(1+M^*\big)+ \pi^{\bar\mu}(|P_t^{\mu}(f)- P_t^{\bar\mu}(f)|),
\end{split}
\end{equation}
where we employed  Assumption  $({\bf A}_{31})$ in the third inequality. To deal with  the second term in the last line of \eqref{WE}, we define the stopping time for any $n\ge1,$
 $$
\tau_n^\cdot =\big\{t\ge0: |Y_t^{\mu,\cdot}|\vee |Y_t^{\bar\mu,\cdot}|>n\big\}.
$$
Then,  we deduce  that
 for any $n,m\ge1,$
\begin{equation}\label{EW2}
\begin{split}
\pi^{\bar\mu}(|P_t^{\mu}(f)- P_t^{\bar\mu}(f)|)
&\le \|f\|_{\rm{Lip}}\pi^{\bar\mu}\big(\E|Y_{t\wedge\tau_n^\cdot}^{\mu,\cdot}
-Y_{t\wedge\tau_n^\cdot}^{\bar\mu,\cdot}|\big)+2\|f\|_\8\pi^{\bar\mu}(\P(t\ge \tau_n^\cdot))\\
&\le \|f\|_{\rm{Lip}}\pi^{\bar\mu}\big(\E|Y_{t\wedge\tau_n^\cdot}^{\mu,\cdot}
-Y_{t\wedge\tau_n^\cdot}^{\bar\mu,\cdot}| \big)\\
 &\quad+2\|f\|_\8\Big(\sup_{|z|\le m}\P(t\ge \tau_n^z)+m^{-\beta^*}\pi^{\bar\mu}(|\cdot|^{\beta^*})\Big),
\end{split}
\end{equation}
where in the second inequality we used   Chebyshev's inequality to derive   the term $m^{-\beta^*}\pi^{\bar\mu}(|\cdot|^{\beta^*})$.

Obviously, we have
\begin{align*}
\d(Y_t^{\mu,x}-Y_t^{\bar\mu,x})=\big(b(Y_t^{\mu,x},\mu)-b(Y_t^{\bar\mu,x},\bar\mu)\big)\,\d t.
\end{align*}
Subsequently, the chain rule and \eqref{E9-} enable us to deduce  that    for any $\vv>0$ and $t<\tau_n^x,$
\begin{equation*}
\begin{split}
\d (\vv+|Y_t^{\mu,x}-Y_t^{\bar \mu,x}|^2)^{\frac{1}{2}}&=\frac{\<Y_t^{\mu,x}-Y_t^{\bar\mu,x},b(Y_t^{\mu,x},\mu)-b(Y_t^{\bar\mu,x},\bar\mu)\>}
{(\vv+|Y_t^{\mu,x}-Y_t^{\bar\mu,x}|^2)^{\frac{1}{2}}}\,\d t\\
&\le \big(K_{n,\mu}(\vv+|Y_t^{\mu,x}-Y_t^{\bar\mu,x}|^2)^{\frac{1}{2}}+ |b(Y_t^{\bar\mu,x},\mu)-b(Y_t^{\bar\mu,x},\bar\mu)|\big)\d t.
\end{split}
\end{equation*}
Whereafter, applying Gronwall's inequality followed by taking $\vv\downarrow0$ yields that
\begin{equation}\label{EW4}
\begin{split}
\E|Y_{t\wedge\tau_n^x}^{\mu,x}-Y_{t\wedge\tau_n^\cdot}^{\bar\mu,x}|&\le \e^{K_{n,\mu}t}\int_0^t \E|b(Y_{s\wedge\tau_n^x}^{\bar\mu,x},\mu)-b(Y_{s\wedge\tau_n^x}^{\bar\mu,x},\bar\mu)| \d s\\
&\le t\e^{K_{n,\mu}t}\sup_{z\in\R^d,|z|\le n}|b(z,\mu)-b(z,\bar\mu)|.
\end{split}
\end{equation}

Next,
  set $Z_t^x:=|Y_t^{\mu,x}|\vee |Y_t^{\bar\mu,x}|$ for notational brevity.
Note that
\begin{align*}
\{\tau_n^x\le t\}=\Big\{\tau_n^x\le t,\sup_{s\in[0,\tau_n^x]}Z_s^x\ge n/2\Big\} +\Big\{\tau_n^x\le t,\sup_{s\in[0,\tau_n^x]}Z_s^x< n/2\Big\},
\end{align*}
where the second event on the right hand side is empty in terms of the definition of $\tau_n^x$. Therefore,  the following implication:
\begin{align*}
\{\tau_n^x\le t\}\subseteq\Big\{ \sup_{s\in[0,t\wedge\tau_n^x]}Z_s^x\ge n/2\Big\}=\big\{ Z_{\tau_{t,n}^{*,x}}^{x}\ge n/2\big\}
\end{align*}
is valid, where
$$
\tau_{t,n}^{*,x}:=t\wedge\tau_n^x\wedge\inf\{s\ge0: Z_s^x\ge n/2\}.
$$
As a consequence, applying Chebyshev's  inequality, in addition to  the increasing property of $[0,\8)\ni r\mapsto (1+r^2)^{\frac{1}{2}\beta^*}$,
gives us that
\begin{align}\label{E15}
\P(t\ge \tau_n^x)\le \frac{2^\beta}{n^\beta}\Big( \E V_\beta\big(Y_{\tau_{t,n}^{*,x}}^{\mu,x}\big)+\E V_\beta\big(Y_{\tau_{t,n}^{*,x}}^{\bar\mu,x}\big)\Big) .
\end{align}
In particular, by taking
\begin{align*}
r_0=1,\quad \vv_1=2^{-\frac{3+\theta_1}{2}}\lambda_{1}/\lambda_{2}  \quad \vv_2=2^{-\frac{7+\theta_1}{2}}\lambda_{1}
\end{align*} and $l$ large enough so that
$$2^{\frac{\beta}{2}}\nu(|\cdot|^{\frac{\beta}{2}}\I_{\{|\cdot|>l\}})\le 2^{-\frac{7+\theta_1}{2}}\beta\lambda_1,$$
we deduce from   \eqref{E10}  that
\begin{equation}\label{e:add}\begin{split}
(\mathscr L_\mu V_\beta)(x)\le& -  2^{-\frac{5+\theta_1}{2}}\beta\lambda_1 V_{\beta^* }(x)  +\Phi(\vv_2,1,l) +  \beta\lambda_2\Gamma(\gamma_1,\gamma_1/\vv_1,\gamma_1)\mu(|\cdot|^{\theta_3})^{\frac{\theta_4}{1-\gamma_1}},
\end{split}
\end{equation}
where $\mathscr L_\mu$ is the infinitesimal generator of $(Y_t^\mu)_{t\ge0}$.
Thus, for
any $(\mathscr F_t)_{t\ge0}$-stopping time $\tau$, the   estimate above implies that
\begin{align*}
\E V_\beta(Y_{t\wedge\tau}^{\mu,x})\le  V_\beta(x)+t\Big(\Phi( \vv_2,1,l) +  \beta\lambda_2\Gamma(\gamma_1,\gamma_1/\vv_1,\gamma_1)\mu(|\cdot|^{\theta_3})^{\frac{\theta_4}{1-\gamma_1}}\Big).
\end{align*}
 This,   in addition to     \eqref{E15},  yields  that for all $x\in \R^d$,
\begin{equation}\label{EW3}
\begin{split}
\P(t\ge \tau_n^x)\le \frac{2^\beta}{n^\beta}\Big(&2V_\beta(x)+t\Big(2\Phi(\vv_2,1,l)\\
&
 +  \beta\lambda_2\Gamma(\gamma_1,\gamma_1/\vv_1,\gamma_1)\Big(\mu(|\cdot|^{\theta_3})^{\frac{\theta_4}{1-\gamma_1}}+
 \bar\mu(|\cdot|^{\theta_3})^{\frac{\theta_4}{1-\gamma_1}}\Big)\Big)\Big) .
 \end{split}
\end{equation}

Now, by plugging \eqref{EW4} and \eqref{EW3} back into  \eqref{EW2},  making use of  $\theta_3\theta_4\le \beta^*(1-\gamma_1)$ as well as applying
 H\"older's inequality,  we derive from Lemma \ref{lem2} that for   $\mu,\bar\mu\in\mathscr P_{\beta^*}^{M^*}$,
\begin{align*}
|\pi^{\mu}(f)-\pi^{\bar\mu}(f)|
&\le  h_{\mu}(t)\|f\|_\8(1+M^*)\\
&\quad+\|f\|_{\rm{Lip}}t\e^{K_{n,\mu}t}\sup_{z\in\R^d,|z|\le n}|b(z,\mu)-b(z,\bar\mu)| +2\|f\|_\8 M^*m^{-\beta^*}\\
&\quad+\frac{2^{2+\beta}\|f\|_\8}{n^\beta}\big((1+m^2)^{\frac{1}{2}\beta
}+ t \big(\Phi( \vv_2,1,l) +  \beta\lambda_2 M^*\Gamma(\gamma_1,\gamma_1/\vv_1,\gamma_1) \big) \big).
\end{align*}
 At last, the  assertion \eqref{EW1} follows by taking \eqref{EW} into account, recalling $\lim_{t\to\8} h_{\mu}(t)=0$,
 and then approaching subsequently $n\uparrow\8,$ $m\uparrow\8$ and $t\uparrow\8$.
\end{proof}

At the end, we present
remarks on verifying Assumption (${\bf A}_{31}$)
and Assumption (${\bf A}_2$), respectively.

\begin{remark}
\begin{itemize}
\item[{\rm (1)}] According to \eqref{e:add}, one can see that under Assumption $({\bf A}_1)$, the so-called Foster-Lyapunov condition holds for the process $(Y_t^\mu)_{t\ge0}$ by taking the Lyapunov function $V_\beta$. In particular, if the process $(Y_t^\mu)_{t\ge0}$ is open-set irreducible and aperiodic, then the Foster-Lyapunov condition above can yield the exponential ergodicity of the process $(Y_t^\mu)_{t\ge0}$; when $\theta_1\in [1-\beta/2,1)$, the condition would imply the sub-exponential ergodicity of the process $(Y_t^\mu)_{t\ge0}$ in case of $\theta_1\ge1$. Hence, Assumption (${\bf A}_{31}$) is satisfied. The readers can refer to \cite{DFG, MT} for the theories of  Foster-Lyapunov method, and consult  \cite{LS} and references therein for recent development.

    \item[{\rm(2)}] As mentioned in the beginning of the proof for Lemma \ref{lem2}, the process $(Y_t^\mu)_{t\ge0}$ associated with \eqref{EW5} is non-explosive under Assumption $({\bf A}_1)$. Furthermore, according to \cite[Proposition 2.1]{Wang2010}, the process $(Y_t^\mu)_{t\ge0}$ enjoys the $C_b$-Feller property, when \eqref{E9-} in Assumption $({\bf A}_{32})$ is satisfied.
\end{itemize}\end{remark}

\section{Proofs of Theorem \ref{thm3}, Example \ref{exm} and example \ref{exm2}}\label{sec3}
In this section, on the basis of Theorem \ref{thm1},
we first carry out  the proof of Theorem \ref{thm3} still via Schauder's fixed point theorem.
\begin{proof}[Proof of Theorem $\ref{thm3}$]
The proof of Theorem \ref{thm3} is essentially inspired by that of \cite[Theorem 3.1]{Zhang}.
According to Theorem \ref{thm1}, \eqref{E1} has a stationary distribution in $\mathscr P_{\beta^*}^{M^*}$ for a suitable   constant $M^*>0$ since $\mathscr P_{\beta^*}^{M^*}$  is an invariant set  of the mapping $\Lambda_\cdot$; see Lemma \ref{lem2} for more details.
By inspecting the procedure of  the proof  for  Theorem \ref{thm1}, to prove that \eqref{E1} has a stationary distribution in $\mathscr P_{\beta^*,y}^{ M^*,M_*}$,
 it remains  to show that $$\mathscr P_{\beta^*,y}^{M_*}:=\{\mu\in\mathscr P_{\beta^*}  : \mu(|y-\cdot|)\le M_*\}$$ is an invariant set of the mapping $\Lambda_\cdot$  since the rest of the proof  is almost identical.

In the sequel, we fix $\mu\in \mathscr P_{\beta^*,y}^{M_*}$.
By It\^o's formula, it follows that
\begin{align*}
\d V_{\vv,\beta,y}(Y_t^\mu) &=  \beta V_{\vv,\beta-2,y}(Y_t^\mu)\<Y_t^\mu-y,b(Y_t^\mu,\mu)\>\,\d t \\
&\quad +\int_{\{|z|\le 1\}}\big\{V_{\vv,\beta,y}(Y_t^\mu+  z)-V_{\vv,\beta,y}(Y_t^\mu)- \<\nn V_{\vv,\beta,y}(Y_t^\mu),z\> \big\}\,\nu(\d z)\,\d t \\
&\quad  +\int_{\{|z|>1\}}\big\{V_{\vv,\beta,y}(Y_t^\mu+  z)-V_{\vv,\beta,y}(Y_t^\mu) \big\}\,\nu(\d z)\, \d t+\d M_t^{\vv,\mu}\\
&=:\big(I_{1,\vv}(t)+I_{2,\vv}(t)+I_{3,\vv}(t)\big)\,\d t+\d M_t^{\vv,\mu},
\end{align*}
where $  (M_t^{\vv,\mu})_{t\ge0}$ is a martingale.
As done previously (see in particular \eqref{E8} and the proof of \eqref{E11}), we can deduce that
\begin{align*}
 I_{2,\vv}(t)+ I_{3,\vv}(t) \le2^{\frac{\beta}{2}}\nu(|\cdot|^{\frac{\beta}{2}}\I_{\{|\cdot|>1\}})|Y_t^\mu-y|^{\frac{\beta}{2}}+ \frac{1}{2}\beta\vv^{\frac{1}{2}(\beta-2)}\nu(|\cdot|^2\I_{\{|\cdot|\le 1\}})+ \nu(|\cdot|^{\beta}\I_{\{|\cdot|>1\}}).
\end{align*}
This, together with \eqref{S6}, yields that
\begin{equation*}
\d V_{\vv,\beta,y}(Y_t^\mu)\le -g_{\vv }(|Y_t^\mu-y|,\mu(|y-\cdot|))\,\d t+\d M_t^{\vv,\mu}.
\end{equation*}
Recall that $g_{\vv }$
 is convex with respect to the first variable.
Thus,   taking $Y_0^\mu=y$ and using Jensen's inequality twice
enables us to derive that
\begin{align*}
\frac{1}{t}\E V_{\vv,\beta,y}(Y_t^\mu)&\le \frac{\vv^{\frac{1}{2}\beta}}{t}-\E\bigg(\frac{1}{t}\int_0^tg_{\vv}(|Y_s^\mu-y|,\mu(|y-\cdot|)\,\d s \bigg)\\
&\le \frac{\vv^{\frac{1}{2}\beta}}{t}-g_{\vv}\bigg(\frac{1}{t}\int_0^t\E |Y_s^\mu-y|\,\d s,  \mu(|y-\cdot|)\bigg).
\end{align*}
This,
along with the prerequisite that $g_{\vv }$
 is continuous with respect to the first variable and the fact that
 $\Lambda_\mu$ is a weak limit of the measure $\frac{1}{t}\int_0^t (\delta_yP_s^\mu)(\cdot)\,\d s$  as $t\to \infty$,
 implies that
\begin{align*}
g_{\vv }\big(\Lambda_\mu(|y-\cdot|), \mu(|y-\cdot|)\big) \le0.
\end{align*}
Since, for fixed $r_1\in[0,\8), $  $[0,\8)\ni r_2\mapsto g_{\vv }(r_1,r_2)$ is decreasing,    we obtain that for $\mu\in \mathscr P_{\beta^*, y}^{M_*}$,
\begin{align*}
g_{\vv }\big(\Lambda_\mu(|y-\cdot|), M_*\big) \le0.
\end{align*}
As a consequence,  due to the precondition that $g_{\vv}(r,M_*)>0$ for all $r>M_*$, we arrive at  $\Lambda_\mu(|y-\cdot|)\le M_*$ so that
$\mathscr P_{\beta^*, y}^{M_*}$ is an invariant set of $\Lambda_\cdot.$

Now we proceed to verify the multiplicity of stationary distributions.
Let $\mu_i\in\mathscr P_{\beta^*, y_i}^{ M^*,M_*}$, $1\le i\le k$,   be  stationary distributions of \eqref{E1}. If the statement
\begin{align}\label{S13}
\mu_i(|\cdot-y_i|
\ge |y_i-y_j|/2)<1/2,\quad 1\le i\neq j\le k
\end{align}
is valid,  then
 for all $1\le i\neq j\le k$, $\mu_i(|\cdot-y_i|<
|y_i-y_j|/2)\ge1/2$.
That is to say, $\mu_i\neq \mu_j$ for all $1\le i\neq j\le k$, provided that the assertion  in \eqref{S13}
is verified. Indeed, by Chebyshev's inequality,
\eqref{S13} is verifiable   by noting that
\begin{align*}
\mu_i(|\cdot-y_i|\ge|y_i-y_j|/2)
 \le \frac{2M_*}{|y_i-y_j|}<1/2,
\end{align*}
where the first  inequality is available  due to $\mu_i\in\mathscr P_{\beta^*,y_i}^{M_*},$ and the second inequality is true
by taking the assumption that $  M_*<\frac{1}{4}\min_{1\le i\neq j\le k}|y_i-y_j|$ into account.
\end{proof}

Next, we move to
  finish respectively the proofs of Examples \ref{exm} and \ref{exm2}.

\begin{proof}[Proof of Example $\ref{exm}$]
Below, to guarantee that \eqref{E2} holds true,  we take $\beta\in(1,\alpha)$.
For notational brevity, we set
\begin{align*}
b(x,\mu):=-\lambda x(x-a_1)(x-a_2)-\kk\mu(x-\cdot),\quad x\in\R, \,\mu\in \mathscr P_1.
\end{align*}
It is easy to see that there exists a constant $ C_1>0$ such that for any $x \in\R$ and $\mu\in\mathscr P_{\beta+2}$,
 \begin{align}\label{SD}
 xb(x,\mu)\le -\frac{1}{2}\lambda|x|^4+\kk|x|\mu(|\cdot|)+C_1.
\end{align}
Whence, Assumption (${\bf A}_1$)(i) holds true with $\theta_1=3$ and
 $\theta_2=\theta_3=\theta_4=1$. In particular, $\beta^*=\beta+2$ and $\gamma_1=(\beta-1)/(\beta+2)$.
 On the other hand, according to \eqref{SD} and \cite[Theorem 7.4]{XZ}, Assumptions (${\bf A}_2$) and (${\bf A}_{31}$) are satisfied. Owing to the validity of Assumptions (${\bf A}_1$) and (${\bf A}_2$), Lemma \ref{lem2} implies that there exists an $M^*>0$   such that $\mathscr P_{\beta+2}^{M^*}$ is an invariant set of the mapping $\Lambda_\cdot$ defined as in \eqref{S10}.
 Additionally, a direct calculation shows that there exists a constant $K>0$ such that
\begin{align*}
(x-y)(b(x,\mu)-b(y,\mu))\le K|x-y|^2,\quad x,y\in\R,~\mu\in\mathscr P_1.
\end{align*}
Furthermore,
 note that for any $x\in\R$, $\mu,\bar\mu\in\mathscr P_{\beta+2}^{M^*} $ and $n\ge1,$
\begin{align*}
 |b(x,\mu)-b(x,\bar\mu)|
 \le &\kk \int_\R (|z_1-z_2|\wedge n)\pi(\d z_1,\d z_2)\\
 & +\kk \int_\R ( |z_1|+|z_2|)(\I_{\{|z_1|>n/2\}}+\I_{\{|z_2|>n/2\}})\pi(\d z_1,\d z_2),
\end{align*}
where $\pi\in\mathscr C(\mu,\bar\mu)$. Correspondingly, for any $x\in\R$ and $\mu,\bar\mu\in\mathscr P_{\beta+2}^{M^*} $,
\begin{equation}\label{ET}
\begin{split}
 |b(x,\mu)-b(x,\bar\mu)|
 \le &\kk \mathbb W_{1,n}(\mu,\bar\mu) +\kk\big(\mu(|\cdot|\I_{\{|\cdot|>n/2\}})+\bar\mu(|\cdot|)\mu(\I_{\{|\cdot|>n/2\}})\big)\\
 & +\kk\big(\mu(|\cdot|)\bar\mu(\I_{\{|\cdot|>n/2\}})+\bar\mu(|\cdot|\I_{\{|\cdot|>n/2\}}) \big),
\end{split}
\end{equation}
where $\mathbb W_{1,n}$ is the  Wasserstein distance induced by the metric $\rho_n(x-y):=|x-y|\wedge n$ for $ x,y\in\R.$
Thus, Chebyshev's inequality, along with $\mu,\bar\mu\in\mathscr P_{\beta+2}^{M^*} $,  enables us to deduce that Assumption (${\bf A}_3$) is provable.
Based on the analysis above, we conclude
that \eqref{exm-} has a stationary distribution in $\mathscr P_{\beta+2}^{M^*} $ by taking  Theorem \ref{thm1} into consideration.

In the sequel, we fix $(\kk,\lambda,\vv)\in(0,\8)^3$ and $\beta\in(1,\alpha)$ satisfying \eqref{-WE} and \eqref{WE2}.
Notice  that for all $x\in\R$ and $\mu\in\mathscr P_1$,
\begin{align*}
(x-a_1)b(x,\mu)&=-\lambda(x-a_1)^4 -\lambda(2a_1-a_2)(x-a_1)^3  -(\lambda a_1(a_1-a_2)+\kk)(x-a_1)^2  \\
&\quad-\kk(x-a_1)\mu(a_1-\cdot).
\end{align*}
Whence, we have for $\beta\in(1,2),$
\begin{align*}
&\beta(\vv+|x-a_1|^2)^{\frac{1}{2}(\beta-2)}(x-a_1) b(x,\mu)\\
&\le -\lambda\beta|x-a_1|^\beta\big(|x-a_1|^{ 2}     -|2a_1-a_2 ||x-a_1|  +  a_1(a_1-a_2)+\kk/\lambda-1  \big)\\
&\quad+\kk\beta |x-a_1|^{\beta-1}\mu(|a_1-\cdot|)+\vv^{\frac{1}{2}\beta} \beta(\lambda a_1(a_1-a_2)+\kk).
\end{align*}
This subsequently implies that
\begin{align*}
\beta(\vv+|x-a_1|^2)^{\frac{1}{2}(\beta-2)}&(x-a_1) b(x,\mu)+2^{\frac{\beta}{2}}\nu(|\cdot|^{\frac{\beta}{2}}\I_{\{|\cdot|>1\}})|x-a_1|^{\frac{\beta}{2}}+\nu(|\cdot|^{\beta}\I_{\{|\cdot|>1\}})\\
&+ \frac{\beta }{2}\vv^{\frac{1}{2}\beta-1}\nu(|\cdot|^2\I_{\{|\cdot|\le 1\}}) \le -g_{\vv,a_1,a_2}(|x-a_1|,\mu(|a_1-\cdot|)),
\end{align*}
where for $a,b\in\R$ and $r_1,r_2>0$
\begin{align*}
g_{\vv,a,b}(r_1,r_2):&=\lambda\beta r_1^\beta\big(r_1^{ 2}     -|2a -b| r_1  +  a (a -b)+\kk/\lambda-1 \big)-\kk\beta r_1^{\beta-1}r_2 -2^{\frac{\beta}{2}}\nu(|\cdot|^{\frac{\beta}{2}}\I_{\{|\cdot|>1\}})r_1^{\frac{\beta}{2}}\\
&\quad-\vv^{\frac{1}{2}\beta}\beta(\lambda a (a -b)+\kk)-\nu(|\cdot|^{\beta}\I_{\{|\cdot|>1\}})- \frac{\beta }{2}\vv^{\frac{1}{2}\beta-1}\nu(|\cdot|^2\I_{\{|\cdot|\le 1\}}).
\end{align*}
A direct calculation shows that
\begin{equation}\label{WW-}
\begin{split}
\frac{\d^2}{\d r_1^2}g_{\vv,a_1,a_2}(r_1,r_2)=&\lambda\beta r_1^{\beta-2}\Big[(2+\beta)(1+\beta)r_1^2-\beta(1+\beta)|2a_1-a_2|r_1 \\
&\qquad\qquad+(a_1(a_1-a_2)+\kk/\lambda-1)\beta(\beta-1) \Big]\\
&+\kk\beta(\beta-1)(2-\beta ) r_1^{\beta-3}r_2  +\beta 2^{\frac{\beta}{2}-1}\nu(|\cdot|^{\frac{\beta}{2}}\I_{\{|\cdot|>1\}})(1-\beta/2 )r_1^{\frac{\beta}{2}-2}.
\end{split}
\end{equation}
Note obviously  that
\begin{align*}
&(2+\beta)(1+\beta)r_1^2-\beta(1+\beta)|2a_1-a_2|r_1  +(a_1(a_1-a_2)+\kk/\lambda-1)\beta(\beta-1) \\
&= (2+\beta)(1+\beta)\bigg(\Big(r_1-\frac{\beta |2a_1-a_2| }{2( 2+\beta)}\Big)^2\\
&\quad\quad\quad\qquad\qquad\qquad+\frac{(a_1(a_1-a_2)+\kk/\lambda-1)\beta(\beta-1)}{(2+\beta)
(1+\beta)}-\frac{\beta^2 |2a_1-a_2|^2 }{4( 2+\beta)^2}\bigg).
\end{align*}
This, together with \eqref{WW-}, leads to the convex property of $r_1\mapsto g_{\vv,a_1,a_2}(r_1,r_2)$ provided that
\begin{align}\label{WW*}
\frac{(a_1(a_1-a_2)+\kk/\lambda-1)\beta(\beta-1)}{(2+\beta)
(1+\beta)}-\frac{\beta^2 |2a_1-a_2|^2 }{4( 2+\beta)^2}\ge0,
\end{align}
which is in turn ensured by  \eqref{-WE}.

 By
interchanging $a_1$ and $a_2$, we can also  conclude that $r_1\mapsto g_{\vv,a_2,a_1}(r_1,r_2)$ is also convex.
It is also easy to see from the arguments above that under \eqref{-WE} all the functions $r_1\mapsto g_{\vv,0,a_1}(r_1,r_2)$, $r_1\mapsto g_{\vv,a_1,0}(r_1,r_2)$, $r_1\mapsto g_{\vv,0,a_2}(r_1,r_2)$ and $r_1\mapsto g_{\vv,a_2,0}(r_1,r_2)$ are convex too.

Below, we fix $(a,b)\in\{(a_1,a_2), (a_2,a_1), (0,a_1), (a_1,0), (0,a_2),(a_2,0)\}$. By virtue of  \eqref{WE2}, we find that
\begin{align*}
g_{\vv,a,b}(0,r_0)<0\quad \mbox{ and } \quad  g_{\vv,a,b}(r_0,r_0)>0.
\end{align*}
This, along with the convex property of $r\mapsto g_{\vv,a,b}(r,r_0)$ (which yields  the uniqueness of the local minimizers), implies that
$g_{\vv,a,b}(r,r_0)>0$ for any $r\ge r_0.$ Furthermore, due to $r_0\in(0,(|a_1|\wedge|a_2|)/4)$ and $a_1a_2<0$, it is ready to see that
\begin{align*}
r_0<\frac{1}{4}(|a_1|\wedge|a_2|)<\frac{1}{4}|a_1-a_2|.
\end{align*}
Consequently, by applying Theorem \ref{thm3} with $M_*=r_0$, \eqref{exm-} has at least three distinct stationary distributions.
\end{proof}

\begin{proof}[Proof of Example $\ref{exm2}$]
Since the  routine  of proof is parallel to that of Example \ref{exm}, in the sequel, we merely outline some of the key points.

 Let for any $x\in\R^d$ and $\mu\in\mathscr P_1$,
\begin{align*}
b(x,\mu)=-\frac{\lambda}{2}\big((x-y_1)|x-y_2|^2+(x-y_2)|x-y_1|^2\big)-\kk\int_{\R^d}(x-y)\,\mu(\d y).
\end{align*}
({\bf a}) We first show that there exists a constant $C_0>0$ such that
\begin{align*}
\<x,b(x,\mu)\>\le -\frac{1}{2}\lambda|x|^4+\kk|x|\mu(|\cdot|)+C_0, \quad x\in\R^d,~\mu\in\mathscr P_1.
\end{align*}
Indeed, this can be achieved by noting that
\begin{align*}
\<x,b(x,\mu)\>
&=-\frac{1}{2}\lambda\big(2|x|^4-2|x|^2\<x,y_1+y_2\>+|x|^2(|y_1|^2+|y_2|^2)\\
&\quad\quad\quad-\<x,y_1+y_2\> |x|^2   +4\<x,y_1\>    \<x,y_2\>  -\<x,y_1|y_2|^2+y_2| y_1|^2 \> \big)\\
&\quad\quad -\kk|x|^2+\kk\mu( \<x,\cdot\>),
\end{align*}
where the other terms are of the lower order with contrast to  the leading term $\lambda|x|^4$.

({\bf b})
Next, we assert that there exists a constant $C_1>0$ such that
\begin{align*}
\<x-y,b(x,\mu)-b(y,\mu)\>\le C_1|x-y|^2,\quad x,y\in\R^d,~\mu\in\mathscr P_1.
\end{align*}
In fact, the above estimate can be established by noticing the following facts:  for all $x,y,z\in\R^d$ and $\mu\in\mathscr P_1, $
\begin{align*}
\<x-y,b(x,\mu)-b(y,\mu)\>&=-\frac{1}{2}\lambda\<x-y,\Gamma(x,y)\>-\kappa|x-y|^2,\\
-\<x-y,|x|^2x-|y|^2y\>&\le -\frac{1}{6}(|x|^2+|y|^2)|x-y|^2,\\
 \big|\<x-y,\<x,z\>x-\<y,z\>y\>\big|
& \le  |x-y|^2|z|(|x|+3|y| ),\\
  \<x-y,z\>\big||x|^2-|y|^2\big|&\le |x-y|^2|z|(|x|+|y|),
\end{align*}
where
\begin{align*}
\Gamma(x,y): &=2(|x|^2x-|y|^2y)-2(\<x,y_2\>x-\<y,y_2\>y)-2(\<x,y_1\>x-\<y,y_1\>y)\\
 &\quad+(|y_1|^2+|y_2|^2)(x-y)-(|x|^2-|y|^2)(y_1+y_2) +2\<x-y,y_1\>y_2+2\<x-y,y_2\>y_1.
\end{align*}

({\bf c})
Subsequently, we prove that for $y\in\{y_1,y_2\}$,
\begin{equation}\label{WQ}
\begin{split}
&\beta(\vv+ |x-y|^2)^{\frac{1}{2}(\beta-2)} \<x-y,b(x,\mu)\> \\
&\quad+ 2^{\frac{\beta}{2}}\nu(|\cdot|^{\frac{\beta}{2}}\I_{\{|\cdot|>1\}})|x-y|^{\frac{\beta}{2}}+\frac{1}{2}\beta\vv^{\frac{1}{2}(\beta-2)}
\nu(|\cdot|^2\I_{\{|\cdot|\le 1\}})+\nu(|\cdot|^{\beta}\I_{\{|\cdot|>1\}})\\
&\le -g_{\vv}(|x-y|,\mu(|y-\cdot|)),
\end{split}
\end{equation}
where for $r_1,r_2>0,$
\begin{align*}
g_{\vv}(r_1,r_2):&=\lambda\beta\Big( r_1^{\beta+2}-\frac{3}{2} r_1^{\beta+1}|y_1-y_2|+\Big(\frac{1}{2}|y_1-y_2|^2+\frac{\kk}{\lambda}-\vv\Big) r_1^\beta\\
&\quad\quad\quad -\Big(\frac{1}{2}|y_1-y_2|^2+\frac{\kk}{\lambda}\Big)\vv^{\frac{\beta}{2}}-\frac{\kk}{\lambda} r_1^{\beta-1}r_2\Big)\\
&\quad- 2^{\frac{\beta}{2}}\nu(|\cdot|^{\frac{\beta}{2}}\I_{\{|\cdot|>1\}})r_1^{\frac{\beta}{2}}-\frac{1}{2}\beta\vv^{\frac{1}{2}(\beta-2)}
\nu(|\cdot|^2\I_{\{|\cdot|\le 1\}})-\nu(|\cdot|^{\beta}\I_{\{|\cdot|>1\}}).
\end{align*}
Actually, \eqref{WQ} can be available due to the fact that  for $y\in\{y_1,y_2\}$,
\begin{align*}
&\beta(\vv+ |x-y |^2)^{\frac{1}{2}(\beta-2)} \<x-y ,b(x,\mu)\>\\
&\le -\lambda\beta\Big( |x-y |^{\beta+2}-\frac{3}{2}|x-y |^{\beta+1}|y_1-y_2|+\Big(\frac{1}{2}|y_1-y_2|^2+\frac{\kk}{\lambda}-\vv\Big) |x-y |^\beta\\
&\quad\quad\quad -\Big(\frac{1}{2}|y_1-y_2|^2+\frac{\kk}{\lambda}\Big)\vv^{\frac{\beta}{2}}-\frac{\kk}{\lambda} |x-y |^{\beta-1}\mu(|\cdot-y|)\Big).
\end{align*}

({\bf d})
Whereafter, we show that $ g_\vv(r ,r_0)\ge0$ for any $r\ge r_0$. To this end, we first verify that
 $r_1\mapsto g_\vv(r_1,r_2)$ is convex. To show this, we find that
\begin{align*}
\frac{\d^2}{\d r_1^2}g_\vv(r_1,r_2)&=\lambda\beta r_1^{\beta-2}\bigg( (\beta+2)(\beta+1)r_1^2-\frac{3}{2}(\beta+1)\beta|y_1-y_2|  r_1\\
&\quad\quad\quad\quad\quad+  \beta(\beta-1)\Big(\frac{1}{2}|y_1-y_2|^2+\frac{\kk}{\lambda}-\vv\Big) \bigg)\\
&\quad+\kk\beta(\beta-1)(2-\beta )r_1^{\beta-3}r_2+ 2^{\frac{\beta}{2}-1} \beta (1-\beta/2)\nu(|\cdot|^{\frac{\beta}{2}}\I_{\{|\cdot|>1\}})r_1^{\frac{\beta}{2}-2}.
\end{align*}
Whence,  the convex property of $r_1\mapsto g_\vv(r_1,r_2)$ can be ensured by \eqref{EQ1}.
Thus, since
\begin{align*}
g_\vv(0,r_0)<0\quad \mbox{ and } \quad g_\vv(r_0,r_0)>0
\end{align*}
by taking advantage of \eqref{WQ2}, the assertion that $ g_\vv(r ,r_0)\ge0$ for any $r\ge r_0$ is provable  by using the convex property of $r\mapsto g_\vv(r,r_0)$.

Based on the analysis above, we conclude that \eqref{WQ2} has at least two stationary distributions by repeating exactly the remaining  details presented in the proof of Example \ref{exm}.
\end{proof}

\section{proofs of theorem  \ref{thm4}, Corollary \ref{cor}  and Example \ref{exa3}}\label{sec4}
This section is mainly devoted to the proof of Theorem   \ref{thm4}, where the ergodicity of $(Y_t^\mu)_{t\ge0}$
under the weighted total variation distance plays an important role. Also,
the proofs of Corollary \ref{cor} and Example \ref{exa3} are to be  finished in this section. In addition,
as far as  $(\mathscr L_{Y_t^\mu})_{t\ge0}$ is concerned,
 Duhamel's formula and gradient estimate are crucial to handle the Lipschitz continuity under the weighted total variation distance
with respect to the frozen distributions.

\begin{proof}[Proof of Theorem $\ref{thm4}$]
Suppose that Assumptions $({\bf A}_1)$ and $({\bf A}_2)$ hold.
By invoking  Lemma \ref{lem2}, there exists an $M^*>0$ such that $\mathscr P_{\beta^*}^{M^*}$ is an invariant subset of the mapping $\Lambda_{\mu}: \mathscr P_{\beta^*}\rightarrow\mathscr P_{{\beta^*}}$ via the relationship $\Lambda_\mu=\pi^\mu$ for $\mu\in\mathscr P_{{\beta^*}}$.

A basic fact is that
 $\mu\in\mathscr P_{\beta^*}^{M^*} $ is a stationary distribution  of \eqref{E1} if and only if $\mu\in\mathscr P_{\beta^*}^{M^*} $ is a fixed point of $\Lambda_\mu$ (i.e., $\Lambda_\mu=\mu$). Based on this,  it amounts to proving that
the mapping $\Lambda_\cdot$ possesses a unique fixed point in order to show  uniqueness of stationary distributions of \eqref{E1}. To achieve the uniqueness of fixed points, we leverage on the Banach's contraction mapping theorem in lieu of Schauder's fixed point theorem.
Since $\Lambda_\cdot$ is a mapping from $\mathscr P_{\beta^*}^{M^*}$ to itself,
to show that $\Lambda_\cdot$ has a unique fixed point in $\mathscr P_{\beta^*}^{M^*}$, it remains to verify that $\Lambda_\cdot$ is contractive, i.e., there exists  $\kk\in(0,1)$ such that for all $\mu_1,\mu_2\in\mathscr P_{\beta^*}^{M^*}$,
\begin{align}\label{WW--}
\|\Lambda_{\mu_1}-\Lambda_{\mu_2}\|_U\le \kk\|\mu_1-\mu_2\|_U.
\end{align}

Below, we shall fix $\mu_1,\mu_2\in \mathscr P_{\beta^*}^{M^*}$. By means of  the invariance of $\Lambda_{\mu_1}$ and $\Lambda_{\mu_2}$,   it follows readily from the triangle inequality  that for any $t\ge0,$
\begin{align}\label{WT1}
\|\Lambda_{\mu_1}-\Lambda_{\mu_2}\|_U\le \|\Lambda_{\mu_1}P_t^{\mu_1}-\Lambda_{\mu_2}P_t^{\mu_1}\|_U
+\|\Lambda_{\mu_2}P_t^{\mu_1}-\Lambda_{\mu_2}P_t^{\mu_2}\|_U.
\end{align}
 Under Assumption
 $({\bf A}_5)$, one has
\begin{align*}
\|\Lambda_{\mu_1}P_t^{\mu_1}-\Lambda_{\mu_2}P_t^{\mu_1}\|_U\le C_{M^*}\e^{-c_{M^*}t}
\|\Lambda_{\mu_1} -\Lambda_{\mu_2} \|_U,\quad t>0.
\end{align*}
This, besides   \eqref{WT1}, leads to
\begin{align}\label{WT4}
\big(1- C_{M^*}\e^{-c_{M^*}t}\big) \|\Lambda_{\mu_1}-\Lambda_{\mu_2}\|_U\le \|\Lambda_{\mu_2}P_t^{\mu_1}-\Lambda_{\mu_2}P_t^{\mu_2}\|_U.
\end{align}
Provided that there exists a locally bounded function $\phi:(0,\8)\to[0,\8)$ such that
\begin{align}\label{WT0}
\|\Lambda_{\mu_2}P_t^{\mu_1}-\Lambda_{\mu_2}P_t^{\mu_2}\|_U\le   K_1\phi_t\|\mu_1-\mu_2\|_U,\quad t>0,
\end{align}
thus, \eqref{WT4} gives   us that
\begin{align*}
\|\Lambda_{\mu_1}-\Lambda_{\mu_2}\|_U\le  c_1 K_1 \|\mu_1-\mu_2\|_U,
\end{align*}
where $K_1$ is the constant in \eqref{WT5} and
\begin{align*}
c_1:=\inf_{t>0: C_{M^*}\e^{-c_{M^*}t}<1}\phi_t \big(1- C_{M^*}\e^{-c_{M^*}t}\big)^{-1}<\infty.
\end{align*}
Accordingly, \eqref{WW--} is valid  once
   $K_1>0$ is enough  small such that   $\kk:=c_1K_1\in(0,1).$

In the sequel,  it remains to verify the statement  \eqref{WT0}. Note that  for all $t>0$,
\begin{align*}
\|\Lambda_{\mu_2}P_t^{\mu_1}-\Lambda_{\mu_2}P_t^{\mu_2}\|_U
&=\sup_{|\psi|\le U}\Big|\int_{\R^d}\big((P_t^{\mu_1}\psi)(x)-(P_t^{\mu_2}\psi)(x)\big) \Lambda_{\mu_2} (\d x)\Big|.
\end{align*}
Furthermore, we can derive from \cite[Theorem 1]{Y} that
\begin{equation}\label{WT2}
\begin{split}
(P_t^{\mu_1}\psi)(x)-(P_t^{\mu_2}\psi)(x)
&=\int_0^tP_{t-s}^{\mu_2}\big((\mathscr L_{\mu_1}(P_s^{\mu_1}\psi)) -(\mathscr L_{\mu_2}(P_s^{\mu_1}\psi))  \big)(x)\,\d s\\
&=\int_0^t\int_{\R^d}\<\nn(P_s^{\mu_1}\psi)(y),b(y,\mu_1)-b(y,\mu_2)\>\,P_{t-s}^{\mu_2}(x,\d y)\,\d s\\
&\le K_1\|\mu_1-\mu_2\|_U\int_0^t\int_{\R^d}\big|\nn(P_s^{\mu_1}\psi)(y)\big|V_{\beta\wedge\beta^*-\beta_0}(y) \,P_{t-s}^{\mu_2}(x,\d y)\,\d s,
\end{split}
\end{equation}
where we utilized \eqref{WT5}  in the last display.

For given  $n\ge1,$  set  $\psi_n(x):=n\wedge \psi(x)$.  Assumption (${\bf A}_6$) and the fact that
 for any   $t\ge0$ and $x,y\in\R^d,$
\begin{align*}
(P_t^{\mu_1}\psi_n)(y)-(P_t^{\mu_1}\psi_n)(x)
 =\int_0^1\<\nn (P_t^{\mu_1}\psi_n)(x+s(y-x)),y-x\>\,\d s
\end{align*}
yield  that with $p=(\beta\wedge\beta^*)/\beta_0>1$
\begin{align}\label{WT7}
\big|(P_t^{\mu_1}\psi_n)(y)-(P_t^{\mu_1}\psi_n)(x)\big|
 \le \frac{\varphi_p(t)}{(1\wedge t)^{\theta}}\int_0^1\big(\E V_{\beta_0 p}(Y_t^{\mu_1,\cdot}(x+s(y-x)))\big)^{{1}/{p}} \,\d s\,|y-x|,
\end{align}
in which we also used the prerequisite $|\psi|  \le U= V_{\beta_0} $   in the inequality.
On the other hand, according to \eqref{E10}, we can find that there is a constant $c_2>0$ so that
\begin{align}\label{WT3}
 \E V_{ \beta}(Y_t^{\mu_1,x}) \le   V_{ \beta}(x)+ c_2t,\quad x\in\R^d, t\ge0.
\end{align}  In particular,
 \begin{equation}\label{WT3-}\E V_{ \beta\wedge \beta^*}(Y_t^{\mu_1,x}) \le  ( V_{ \beta}(x)+ c_2t)^{(\beta\wedge\beta^*)/\beta}\le  V_{ \beta\wedge\beta^*}(x)+ (c_2t)^{(\beta\wedge\beta^*)/\beta} ,\quad x\in\R^d, t\ge0.\end{equation}
This, along with \eqref{WT7} and the dominated convergence theorem,  further implies that for  $x,y\in\R^d$ and $t>0$,
\begin{align*}
|(P_t^{\mu_1}\psi)(y)-(P_t^{\mu_1}\psi)(x)|&\le \frac{\varphi_p(t)}{(1\wedge t)^\theta}\int_0^1\big(V_{\beta\wedge\beta^*}(x+s(y-x))+(c_2t)^{(\beta\wedge\beta^*)/\beta} \big)^{\beta_0/(\beta\wedge \beta^*)} \,\d s\, |y-x|.
\end{align*}
Whence, for all $x\in \R^d$ and $t>0$,
\begin{align}\label{EY}
\big|\nn P_t^{\mu_1}\psi \big|(x) \le \frac{\varphi_p(t)}{(1\wedge t)^{\theta}} \big((c_2t)^{\beta_0/\beta}+V_{\beta_0}(x) \big) .
\end{align}
Plugging the previous estimate  back into \eqref{WT2}  leads to
\begin{equation*}
\begin{split}
&(P_t^{\mu_1}\psi)(x)-(P_t^{\mu_2}\psi)(x)\\
&\le  K_1 \|\mu_1-\mu_2\|_U   \int_0^t \int_{\R^d}\frac{\varphi_p(s)\big((c_1s)^{\beta_0/\beta}+V_{\beta_0}(y)\big)V_{\beta\wedge\beta^*-\beta_0}(y)}{(1\wedge s)^{\theta}} P_{t-s}^{\mu_2}(x,\d y) \d s\\
&\le  K_1\|\mu_1-\mu_2\|_U \Big(1+(c_2t)^{\frac{\beta_0}{\beta}}\Big)\int^t_0 \frac{\varphi_p(s)}{(1\wedge s)^{\theta}} V_{\beta\wedge\beta^*}(y)\,P_{t-s}^{\mu_2}(x,\d y)\,\d s \\
&\le  K_1 \|\mu_1-\mu_2\|_U\Big(1+(c_2t)^{\frac{\beta_0}{\beta}}\Big)\Big(1+(c_2t)^{\frac{\beta\wedge\beta^*}{\beta}}\Big)\int^t_0 \frac{\varphi_p(s)}{(1\wedge s)^{\theta}}  \d s V_{\beta\wedge\beta^*}(x),
\end{split}
\end{equation*} where in the last inequality we used \eqref{WT3-} again.
With this at hand,
we conclude that \eqref{WT0} is available  by taking
\begin{align*}
\phi_t:=   (1+M^*)^{(\beta\wedge\beta^*)/\beta^*}\Big(1+(c_2t)^{\frac{\beta_0}{\beta}}\Big)\Big(1+(c_2t)^{\frac{\beta\wedge\beta^*}{\beta}}\Big) \int_0^t \frac{\varphi_p(s)}{(1\wedge s)^{\theta}}  \d s.
\end{align*}

Below, we fix  $\mu\in\mathscr P_\beta$.
For any $s\ge0$, let $(X_{s,t}^\mu)_{t\ge s}$ be the solution to  \eqref{E1} initiating from  time $s$ with $\mathscr L_{X_{s,s}^\mu}=\mu$, and  define
\begin{align*}
P_{s,t}^*\mu=\mathscr L_{X_{s,t}^\mu},\quad t\ge s.
\end{align*}
In case of $s=0$, we shall write $X_{t}^\mu$ and
$P_{ t}^* $ instead of $X_{0,t}^\mu$ and
$P_{0, t}^* $, respectively, for the sake of notational convenience.
As shown in \cite[(1.10)]{Wang}, $(P_{s,t}^*)_{t\ge s\ge0}$ is a nonlinear semigroup satisfying
\begin{align}\label{WY4}
P_{s,t}^*=P_{r,t}^*P_{s,r}^* \quad \mbox{ and }  \quad P_{s,t}^*=P_{t-s}^*, \quad 0\le s\le r\le t.
\end{align}

Let $\pi$ be the unique stationary distribution of the SDE \eqref{E1}. By the invariance of $\pi$, Assumption (${\bf A}_5$) and
 the triangle inequality, it is easy to see that for any $t\ge0,$
\begin{align*}
\|P_t^*\mu -\pi\|_U&\le \|P_t^*\mu  - (P_t^\pi)^*\mu\|_U+
\| (P_t^\pi)^*\mu -  (P_t^\pi)^*\pi\|_U\\
&\le \|P_t^*\mu  -(P_t^\pi)^*\mu\|_U+
C_{M^*}\e^{-c_{M^*}t}\|\mu -\pi\|_{U},
\end{align*}
in which $(P_t^\pi)^*\mu:=\mu P_t^\pi$.
For  $\hat \mu\in\mathscr P,$ let $( (\hat P_{s,t}^{\mu})^*\hat\mu)_{t\ge s\ge0}$ be  the law of the linear Markov process
$(X_{s,t}^{\mu,\hat \mu})_{t\ge s\ge0}$, which    solves   the following decoupled SDE:
\begin{align}\label{WY7}
\d X_{s,t}^{\mu,\hat\mu}=b(X_{s,t}^{\mu,\hat \mu}, P_{s,t}^*\mu)\,\d t+\d Z_t,\quad t\ge s\ge0
\end{align} with $\mathscr L_{X_{s,s}^{\mu,\hat \mu}}=\hat\mu$.
Since the SDE \eqref{E1} is (weakly) well-posed, we thus have $P_{s,t}^*\mu= ( \hat P_{s,t}^{\mu })^*\mu$.
Again we write
$(\hat P^\mu_{t})^* $ instead of
$(\hat P^\mu_{0,t})^*$ for simplicity. Then,
\begin{align}\label{WY2}
\|P_t^*\mu -\pi\|_U\le \|( \hat P_{t}^{\mu })^*\mu -(P_t^\pi)^*\mu\|_U
+
C_{M^*}\e^{-c_{M^*}t}\|\mu -\pi\|_{U},\quad t\ge0.
\end{align}
Next, by following the strategy to derive \eqref{WT0}, we derive from \eqref{UT}, \eqref{WT5} and \eqref{EY} that
\begin{equation}\label{WY}
\begin{split}
\|( \hat P_{t}^{\mu })^*\mu -(P_t^\pi)^*\mu\|_U&\le \sup_{|\psi|\le U}\bigg|\int_{\R^d}\big((\hat P_t^{\mu }\psi)(x)-(P_t^{\pi}\psi)(x)\big)\mu(\d x)\bigg|\\
&\le  K_1\int_0^t\frac{\varphi_p(s)(1+(c_2s)^{\beta_0/\beta})}{(1\wedge s)^{\theta}}\|P_s^*\mu -\pi\|_U\big((\hat P_{s,t}^{\mu})^* \mu\big)( V_{\beta\wedge\beta^*})\,\d s.
\end{split}
\end{equation}
In addition, by following
the argument for \eqref{e:add}, we can deduce that for some constants $c_3,c_4>0,$
\begin{align*}
\E V_\beta(X_{s,t}^{\mu,x})\le V_\beta(x)-c_3\int_s^t\big(\E V_{\beta^*}(X_{s,u}^{\mu,x})-((P_{s,u}^*\mu)(|\cdot|^{\theta_3}))^{{\theta_4}/({1-\gamma_1})}\big)\,\d u+c_4(t-s),\quad t\ge s\ge0,
\end{align*}
where $X_{s,t}^{\mu,x}:=X_{s,t}^{\mu,\delta_x}$.
This obviously implies that
\begin{align*}
\E V_\beta(X_{s,t}^{\mu, \mu})\le  \mu(V_\beta) +c_4(t-s),\quad t\ge s\ge0,
\end{align*} thanks to the fact that $\theta_3\theta_4\le \beta^*(1-\gamma_1)$ under Assumption $({\bf A}_1)$.
Whence, applying H\"older's inequality,  we   obtain from \eqref{WY2} and \eqref{WY}   that
\begin{equation*}
\begin{split}
\|P_t^*\mu -\pi\|_U
&\le  K_1\varphi(t)\int_0^t\frac{\varphi_p(s)}{(1\wedge s)^{\theta}} \|P_s^*\mu -\pi\|_U\,\d s +
C_{M^*}\e^{-c_{M^*}t}\|\mu -\pi\|_{U},
\end{split}
\end{equation*}
where
\begin{align*}
\varphi(t):=(1+(c_1t)^{\beta_0/\beta})\big(\mu(V_\beta)+c_4t\big)^{({\beta\wedge\beta^*})/{\beta}}.
\end{align*}
Subsequently, the Gronwall inequality (see e.g. \cite[Theorem 11]{Dra}) yields that for all $t\ge0,$
\begin{align*}
\|P_t^*\mu -\pi\|_U\le \|\mu -\pi\|_{U}\bigg[&C_{M^*}\e^{-c_{M^*}t} \\
&+K_1C_{M^*}\varphi(t)\int_0^t \frac{\e^{-c_{M^*}s}\varphi_p(s)}{(1\wedge s)^{\theta}}\exp\bigg(\int_s^t\frac{K_1\varphi(r)\varphi_p(r)}{(1\wedge r)^{\theta}}\,\d r\bigg)\,\d s\bigg].
\end{align*}
Below, we take $t=t_*:=\frac{1}{c_M^*}\log(1\vee(2C_M^*))$ and choose $K^*>0$ such that for any $K_1\in(0,K^*],$
\begin{align*}
\lambda_*:=\frac{1}{2}+K_1C_{M^*}\varphi(t_*)\int_0^{t_*} \frac{\e^{-c_{M^*}s}\varphi_p(s)}{(1\wedge s)^{\theta}}\exp\bigg(\int_s^{t_*}\frac{K_1\varphi(r)\varphi_p(r)}{(1\wedge r)^{\theta}}\d r\bigg)\d s<1.
\end{align*}
Whereafter, we arrive at
\begin{align*}
\|P_{t_*}^*\mu -\pi\|_U\le \lambda_*\|\mu -\pi\|_{U} .
\end{align*}
Therefore, the assertion \eqref{WY3} follows from \eqref{WY4}.
\end{proof}

Next, we turn to the proof of Corollary \ref{cor}.

\begin{proof}[Proof of Corollary $\ref{cor}$]
According to Theorem \ref{thm4} and Remark \ref{re1.8}(1),  it  suffices to show that Assumption (${\bf A}_6$)
is satisfied in order to complete the proof of  Corollary \ref{cor}. To justify  Assumption (${\bf A}_6$), we introduce the   Yoshida approximation
of $b(x,\mu)$ (for fixed $\mu\in\mathscr P_{\beta^*}$), which is defined as below:
\begin{align*}
 b^{(n)}(x,\mu)=n\big((I-b(\cdot,\mu)/n)^{-1}(x)-x\big), \quad x\in\R^d, ~n\ge1,
 \end{align*}
 where $I$ means the identity map, i.e., $I(x)=x$ for $x\in\R^d$.  It is well known  that, for any $x,y\in\R^d,$
 $ |b^{(n)}(x,\mu)|\le |b(x,\mu)|$,
 \begin{align}\label{Y1}
 | b^{(n)}(x,\mu)- b^{(n)}(y,\mu)|\le 2n|x-y|,
\end{align}
and,  by taking  \eqref{WT} into account,
 \begin{align}\label{Y3}
 \<x-y,b^{(n)}(x,\mu)-b^{(n)}(y,\mu)\>\le K_2|x-y|^2;
 \end{align}
 see, for instance,  \cite[Proposition 5.5.3]{DZ} for details.
Correspondingly,   the  SDE
 \begin{align}\label{Y2}
 \d Y_t^{n,\mu}=b^{(n)}(Y_t^{n,\mu},\mu)\,\d t+\d Z_t
 \end{align}
 has a unique strong solution $(Y_t^{n,\mu})_{t\ge0}$. In the sequel, we write $(Y_t^{n,\mu})_{t\ge0}$ as $(Y_t^{n,\mu,x})_{t\ge0}$
if we want to emphasize the initial value $Y_0^{n,\mu}=x\in\R^d$, and set $(P_t^{n,\mu})_{t\ge0} $ to be the Markov semigroup generated by $(Y_t^{n,\mu})_{t\ge0}$.

Recall that $(Z_t)_{t\ge0}$ is a symmetric $\alpha$-stable process, which can be expressed as $(W_{S_t})_{t\ge0}$. Here $(W_t)_{t\ge0}$ is a $d$-dimensional Brownian motion and $(S_t)_{t\ge0}$ is an $\alpha/2$ subordinator independent of $(W_t)_{t\ge0}$.
Since, for fixed $n\ge1,$ $x\mapsto b^{(n)}(x,\mu)$ is Lipschitz (see \eqref{Y1}),
\cite[Theorem 1.1]{Zhangxicheng} implies that for any $f\in C_b (\R^d)$, $t>0$ and $x\in\R^d,$
\begin{align*}
\nn_h(P_t^{n,\mu}f)(x)=\E\bigg(\frac{1}{S_t}f(Y_t^{n,\mu,x})\int_0^t\<\nn_h Y_s^{n,\mu,x}, \d W_{S_s}\> \bigg).
\end{align*}
Hereinabove,   $\nn_h(P_t^{n,\mu}f)(x)$ stands for the directional derivative of  $P_t^{n,\mu}f$ at the point $x$ along the direction $h$,  and the directional derivative $(\nn_h Y_t^{n,\mu,x})_{t\ge0}$ is determined by  the following  ODE
\begin{align*}
\d \nn_h Y_t^{n,\mu,x}=\nn b^{(n)}(Y_t^{n,\mu,x},\mu)\nn_h Y_t^{n,\mu,x}\,\d t,\quad t>0; \quad \nn_h Y_0^{n,\mu,x}=h.
\end{align*}
By the aid of \eqref{Y3},  it follows readily that
\begin{align*}
\<v,\nn_v b^{(n)}(z,\mu)\>\le K_2|v|^2, \quad v,z\in\R^d,
\end{align*}
so that   the chain rule  yields that
\begin{align*}
\d|\nn_h Y_t^{n,\mu,x}|^2\le K_2|\nn_h Y_t^{n,\mu,x}|^2\,\d t.
\end{align*}
This apparently leads to the following estimate:
\begin{align*}
 |\nn_h Y_t^{n,\mu,x}|^2 \le  \e^{ K_2t}|h|^2 .
\end{align*}
Subsequently, for any $p>1$, by tracking the strategy to derive \cite[(1.7)]{Zhangxicheng},
we derive from H\"older's inequality and the  BDG inequality that there exists a constant $C_p>0$ such that for  $f\in C_b(\R^d)$, $t>0$ and $x\in\R^d,$
\begin{align*}
|\nn(P_t^{n,\mu}f)(x)|\le C_p\e^{\frac{1}{2}K_2t}(1\wedge t)^{-\frac{1}{\alpha}} (P_t^{n,\mu}|f|^p)^{\frac{1}{p}}(x).
\end{align*}
 Whereafter,  applying the identity
\begin{align*}
P_t^{n,\mu}f(x)-P_t^{n,\mu}f(y)=\int_0^1\<\nn (P_t^{n,\mu }f)(x+s(y-x)),y-x\>\,\d s,
\end{align*}
and combining the fact that   $\lim_{n\to\8}Y_t^{n,\mu}=Y_t^\mu$ a.s. (see e.g.  \cite[p.520]{WW}) with  the dominated convergence theorem yields that
\begin{align*}
|\nn(P_t^{\mu}f)(x)|\le C_p\e^{\frac{1}{2}K_2t}(1\wedge t)^{-\frac{1}{\alpha}} (P_t^{\mu}|f|^p)^{\frac{1}{p}}(x).
\end{align*}
As a consequence, $({\bf A}_6)$ holds true with $\varphi_p(t)=C_p\e^{\frac{1}{2}K_2t}$ and $\theta=1/\alpha.$
\end{proof}

\begin{proof}[Proof of Example $\ref{exa3}$]
For $x\in\R$ and $\mu\in\mathscr P_1,$ set
\begin{align*}
b(x,\mu)=-\bigg( \lambda x(x-1)(x+2) -\kappa\int_\R((1+|x|^2)^{\frac{1}{2}(\beta-1)}|y|+g(y))\mu(\d y)\bigg).
\end{align*}
A direct calculation shows that
\begin{align*}
xb(x,\mu)\le-\lambda|x|^2 (x^2+x-2) +\kappa\big( (1+ |x|^2)^{\beta/2 }\mu(|\cdot|)+|x|\cdot\|g\|_\8\big).
\end{align*}
Thus, Assumption $({\bf A}_1) (i)$ (so Assumption $({\bf A}_{53})$ in the Appendix section) holds true with $\lambda_1=\lambda/2$, $\theta_1=3,$ $\theta_2=\beta$, $\theta_3=\theta_4=1$ and $\lambda_2=\kappa. $ In particular, $\beta^*=\beta+2.$
Next,
by invoking the fact that $[0,\8)\ni r\mapsto(1+r^2)^{\frac{1}{2}(\beta-1)}$ has a bounded derivative, there exists a constant $C_0>0$ such that for all $x,y\in\R$ and $\mu\in\mathscr P_{\beta+2}^{M^*}$,
 \begin{align*}
(x-y)(b(x,\mu)-b(y,\mu))\le C_0|x-y|^2
 \end{align*}
 so that \eqref{E9-} and Assumption $({\bf A}_{52})$ are  satisfied, respectively.
 In addition, \eqref{EW} is verifiable by using the fact that $g$ is bounded and continuous, and mimicking the related details carried out in the proof of Example \ref{exm}; see in particular
\eqref{ET}. Therefore, we conclude that Assumption $({\bf A}_{32})$ is provable.

Recall that $(Z_t)_{t\ge0}
$ is a symmetric $\alpha$-stable process. Thus,     \cite[Example 1.2]{LW} demonstrates that Assumption $({\bf A}_{51})$  is valid and \eqref{E2} is satisfied
 since the L\'{e}vy measure $\nu(\d z)$ of $(Z_t)_{t\ge0}$ admits the expression  $\nu(\d z)=\frac{c}{|z|^{1+\alpha}}\,\d z$ for some   $c>0.$ As a consequence of Proposition \ref{pro}, \eqref{UT} in Assumption $({\bf A}_{5})$ holds. In particular, $({\bf A}_{31})$ is fulfilled.  In turn, Proposition \ref{pro1} enables us to deduce that Assumption  $({\bf A}_{3 })$ is available.

 In a word,  by applying Corollary \ref{cor}, we infer that \eqref{ETY} has a unique stationary distribution.
\end{proof}

\section{Appendix}

\setcounter{equation}{0}

\renewcommand\theequation{A\arabic{equation}}

\subsection{Uniqueness of stationary distributions via the $L^1$-Wasserstein distance}
In this subsection, we give another collection of sufficiency to check the uniqueness of stationary distributions. More precisely, we assume that

\begin{enumerate}\it
\item[{\rm(${\bf H}_1$)}] there exist constants $\ell_0, K_1,K_2,K_3>0$ such that for all $x,y\in\R^d$ and $\mu_1,\mu_2\in\mathscr P_1,$
\begin{align*}
\<x-y,b(x,\mu_1)-b(y,\mu_2)\>\le\big(K_1 \I_{\{|x-y|\le \ell_0\}}-K_2\I_{\{|x-y|>\ell_0\}}\big)|x-y|^2+K_3\mathbb W_1(\mu_1,\mu_2)^2.
\end{align*}

\item[{\rm(${\bf H}_2$)}]there is a positive, non-decreasing and concave function  $\si\in C([0,2\ell_0])\cap C^2((0,2\ell_0])$ such that
\begin{align*}
\sigma(r)\le \frac{1}{2r}J(\kk\wedge r)(\kk\wedge r)^2,\quad r\in(0,2\ell_0],
\end{align*} where
  $$J(r):=\inf_{x\in\R^d,|x|\le r}\nu_x(\R^d),\quad  r>0$$ with $$  \nu_x(\d z):=(\nu\wedge(\delta_x*\nu))(\d z).$$
\end{enumerate}

Theorem \ref{thm4}  shows that \eqref{E1} has a
 unique  stationary distribution  when $\mu\mapsto b(\cdot,\mu)$ is continuous  under the weighted total variation distance with the interaction intensity small enough. The theorem below reveals that \eqref{E1} can also possess a
 unique  stationary distribution when $\mu\mapsto b(\cdot,\mu)$ is continuous under the $L^1$-Wasserstein distance and the corresponding interaction intensity is small (i.e., the constant $K_3$ in Assumption {\rm(${\bf H}_1$)} is small).

\begin{theorem}\label{thm2}
Assume that  $({\bf H}_1 )$ and $({\bf H}_2)$ hold, and that \eqref{E2} is satisfied with $\beta_*\in[1,2]$.
hold. Then, there exists a constant $K^*>0$ such that, for any $K_3\in(0,K^*],$
\eqref{E1} has a unique stationary distribution $\pi\in\mathscr P_1$, and moreover there exist constants $C_0,\lambda>0$ such that for all $t\ge0$ and $\mu\in\mathscr P_1,$
\begin{align}\label{WY10}
\mathbb W_1(P_t^*\mu,\pi)\le C_0\e^{-\lambda t}\mathbb W_1( \mu,\pi),
\end{align} where $P_t^*\mu:=\mathscr L_{X_t}$  with
 the initial distribution $\mathscr L_{X_0}=\mu.$
\end{theorem}

\begin{remark} We make some remarks  on  Assumptions (${\bf H}_1$)  and (${\bf H}_2$), and compare Theorem \ref{thm2} with the existing literature \cite{LMW}.
\begin{itemize}
\item[{\rm(1)}] Notice that  (${\bf H}_1$) and (${\bf H}_2$)  are isolated since there are no mutual implications between them.
The role of (${\bf H}_1 $) is two-fold. Under (${\bf H}_1 $) with $\mu_1=\mu_2$  and (${\bf H}_2$), the reference SDE \eqref{EW5} is exponentially ergodic under the $L^1$-Wasserstein distance; see, for instance,   \cite[Theorem 4.2]{LW}. This is indispensable to demonstrate the uniqueness of stationary distributions. On the other hand,  (${\bf H}_1 $)  is exploited  to show the $L^1$-Wasserstein
continuity of $(Y_t^{\cdot})_{t\ge0}$  with respect to the frozen distributions; see, in particular,  \eqref{S3} below for more details.

\item[{\rm(2)}]
In \cite[Theorem 1.4]{LMW}, \eqref{E1} is shown to have a unique stationary distribution under $({\bf H}_2)$, (${\bf H}_1'$) as well as
\eqref{E2} with $\beta_*\in[1,2)$, where
\begin{enumerate}
\item[(${\bf H}_1'$)] {\it
there exist constants $\ell_0, K_1,K_2,K_3>0$ such that  for all $x,y\in\R^d$ and $\mu_1,\mu_2\in\mathscr P_1,$
\begin{align*}
\<x-y,b(x,\mu_1)-b(y,\mu_2)\>\le&\big(K_1 \I_{\{|x-y|\le \ell_0\}}-K_2\I_{\{|x-y|>\ell_0\}}\big)|x-y|^2\\
&+K_3|x-y|\mathbb W_1(\mu_1,\mu_2).
\end{align*}}
\end{enumerate}
It is obvious  that  (${\bf H}_1'$) is stronger than (${\bf H}_1$).
In the proof of \cite[Theorem 1.4]{LMW}, the $L^1$-Wasserstein contraction of $(X_t)_{t\ge0}$ (rather than $(Y_t^\mu)_{t\ge0}$) plays an important role.
However, as long as  (${\bf H}_1'$) is substituted by (${\bf H}_1$),  the strategy therein leveraging on  the reflection coupling used in the proof of \cite[Theorem 1.4]{LMW} no longer  works  due to  the appearance of the term $\mathbb W_1(\mu_1,\mu_2)^2$ (instead of $|x-y|\mathbb W_1(\mu_1,\mu_2)$), which can bring  about the singularity on the diagonal by applying the coupling method. Based on such a   viewpoint,  Theorem \ref{thm2} indicates that
\eqref{E1} might enjoy  a unique stationary distribution under the weaker Assumption (${\bf H}_1$) with a weak interaction.\end{itemize}
\end{remark}

The proof of Theorem \ref{thm2}  relies  on the ergodicity of $(Y_t^\mu)_{t\ge0}$
and the  $L^1$-Wasserstein Lipschitz continuity in a finite-time interval with respect to the frozen distributions.

\begin{proof}[Proof of Theorem $\ref{thm2}$]
Under (${\bf H}_1$) with $\mu_1=\mu_2$ and (${\bf H}_2$),  for each fixed $\mu\in\mathscr P_1$,  we deduce from \cite[Theorem 4.2]{LW}  that
\begin{align}\label{S1}
\mathbb W_1(\mu_1P_t^\mu,\mu_2P_t^\mu)\le C\e^{-\lambda t}\mathbb W_1(\mu_1,\mu_2),\quad \mu_1,\mu_2\in\mathscr P_1,
\end{align}
where
\begin{align*}
C:&=\frac{1}{2}(1+1/{c_1})=\frac{1}{2}\big(1+\exp\big(g(2\ell_0)((2K_2)\wedge g_1(2\ell_0)^{-1})\big)\big),\\
\lambda:&=\frac{c_2}{1+\e^{2g(2\ell_0)}}=\frac{(2K_2)\wedge g_1(2\ell_0)^{-1}}{1+\exp(g(2\ell_0)((2K_2)\wedge g_1(2\ell_0)^{-1}))}.
\end{align*}
Herein, for $g_1(r):=\int_0^r\frac{1}{\sigma(s)}\,\d s$ and $g_2(r):=K_1\int_0^r\frac{1}{\sigma(s)}\,\d s$,
\begin{align*}
c_2:=(2K_2)\wedge g_1(2\ell_0)^{-1},\quad c_1:=\e^{-c_2g(2\ell_0)}\quad \mbox{ and } \quad g(r):=g_1(r)+\frac{2}{c_2}g_2(r),\quad r\in(0,2\ell_0].
\end{align*}
Thus, via  Banach's fixed point theorem and the fact that \eqref{E2} is satisfied with $\beta_*\in[1,2]$, \eqref{S1} shows that, for each fixed $\mu\in\mathscr P_1 $,  $(P_t^\mu)_{t\ge0}$  has a unique IPM  $\pi^\mu\in\mathscr P_1.$
Consequently, we can define the following mapping
\begin{align*}
\Lambda_\cdot: \mathscr P_1(\R^d)\to \mathscr P_1(\R^d),\quad \mu\mapsto \Lambda_\mu:=\pi^\mu.
\end{align*}
As analyzed in the proof of Theorem \ref{thm4}, it is sufficient to prove that there exists a constant $\kk\in(0,1)$ such that
\begin{align}\label{S2}
\mathbb W_1(\Lambda_{\mu_1},\Lambda_{\mu_2})\le \alpha\mathbb W_1(\mu_1,\mu_2)
\end{align}
in order to show the uniqueness of stationary distributions.

Below, we stipulate $\mu_1,\mu_2\in\mathscr P_1$. Due to the invariance of $\Lambda_{\mu_1},\Lambda_{\mu_2}$ (so $\Lambda_{\mu_1}=\Lambda_{\mu_1}P_t^{\mu_1}$ and $\Lambda_{\mu_2}=\Lambda_{\mu_2}P_t^{\mu_2}$ for all $t\ge0$), the triangle inequality and \eqref{S1}
imply that for all $t\ge0,$
\begin{align*}
\mathbb W_1(\Lambda_{\mu_1},\Lambda_{\mu_2})&\le \mathbb W_1(\Lambda_{\mu_1}P_t^{\mu_1},\Lambda_{\mu_2}P_t^{\mu_1})+\mathbb W_1(\Lambda_{\mu_2}P_t^{\mu_1},\Lambda_{\mu_2}P_t^{\mu_2})\\
&\le C\e^{-\lambda t}\mathbb W_1(\Lambda_{\mu_1},\Lambda_{\mu_2})+\mathbb W_1(\Lambda_{\mu_2}P_t^{\mu_1},\Lambda_{\mu_2}P_t^{\mu_2}).
\end{align*}
This subsequently gives that  for all $t\ge0,$
\begin{align}\label{S5}
(1-C\e^{-\lambda t})\mathbb W_1(\Lambda_{\mu_1},\Lambda_{\mu_2})\le \mathbb W_1(\Lambda_{\mu_2}P_t^{\mu_1},\Lambda_{\mu_2}P_t^{\mu_2}).
\end{align}
As soon as  we can claim the following statement:
\begin{align}\label{S3}
\mathbb W_1(\Lambda_{\mu_2}P_t^{\mu_1},\Lambda_{\mu_2}P_t^{\mu_2})\le K_3^{\frac{1}{2}} C_t\mathbb W_1(\mu_1,\mu_2),
\end{align}
where
\begin{align*}
C_t:=2^{\frac{1}{2}} \nu( {B_1^c})t^{\frac{1}{2}}\e^{(K_1-\nu( {B_1^c}))t}\big(1+\nu( {B_1^c})t \e^{K_1t}\big)\e^{\nu( {B_1^c})t\e^{K_1t}} ,\quad t\ge0,
\end{align*}
we obtain from \eqref{S5} that
\begin{align*}
 C_t^{-1} (1-C\e^{-\lambda t})\mathbb W_1(\Lambda_{\mu_1},\Lambda_{\mu_2})\le K_3^{\frac{1}{2}}\mathbb W_1(\mu_1, \mu_2).
\end{align*}
Whence, the assertion \eqref{S2} follows as long as $K_3<(K^*)^2, $ where $$
K^*:=\sup_{t>\ln C/\lambda}(C_t^{-1} (1-C\e^{-\lambda t})).
$$

In the sequel,
we adopt the so-called interlacing technique to establish \eqref{S3} under Assumption (${\bf H}_1$). Let $N(\d s, \d z)$ be a Poisson random
measure with the intensity measure $\d t\times \nu(\d z)$, where $\nu$ is the L\'evy measure corresponding to the L\'evy process $(Z_t)_{t\ge0}$. In particular,
$$\d Z_t=\int_{\{|z|\le 1\}}z\, \hat N(\d t, \d z)+ \int_{\{|z|> 1\}}z\,   N(\d t, \d z),$$ where $\hat N(\d t,\d z)=N(\d t, \d z)-\d t\,\nu(\d z).$
 To this end, we
introduce the process $(Z^*_t)_{t\ge0}$ defined by $Z^*_t=\int_0^t\int_{B_1^c}zN(\d s,\d z)$, where $B_1^c:=\{x\in\R^d: |x|>1\}$. To characterize the jumping time of $(Z^*_t)_{t\ge0}$, we  set
$D_p:=\{t\ge0: Z_t\neq Z_{t-},\Delta Z_t \in B_1^c\},$
where
$\Delta Z_t:=Z_t-Z_{t-}$ means the increment of $(Z_t)_{t\ge0}$ at time $t$.
Note that the set $D_p$ is  countable   so   it can be rewritten countably as
$
D_p=\{\si_1,\cdots,    \si_n,\cdots\},
$
in which the sequence of stopping times $(\si_n)_{n\ge1}$ satisfies that
 $n\mapsto \si_n$ is increasing, and
  $\lim_{n\to\infty}\si_n=\infty$ a.s. by invoking  $\nu(\I_{B_1^c})<\8$.

It is ready to see that for any $t\ge0,$
\begin{align}\label{S4}
\mathbb W_1(\Lambda_{\mu_2}P_t^{\mu_1},\Lambda_{\mu_2}P_t^{\mu_2})=\sum_{n=0}^\8\E\big(|\Upsilon_t^{\mu_1,\mu_2}|\I_{\{\si_n\le t<\si_{n+1}\}}\big).
\end{align}
where $\Upsilon_t^{\mu_1,\mu_2}:=Y_t^{\mu_1,\mu_2}-Y_t^{\mu_2,\mu_2}$. Here, $(Y_t^{\mu,\mu})_{t\ge0}$ denotes the solution to \eqref{EW5} with the initial distribution $\mathscr{L}_{Y_0^\mu}=\mu$.
By the chain rule, it follows from (${\bf H}_1$) that for any $t\in[\si_n,\si_{n+1}),$
\begin{align*}
\d |\Upsilon_t^{\mu_1,\mu_2}|^2&=2\<\Upsilon_t^{\mu_1,\mu_2}, b(Y_t^{\mu_1,\mu_2},\mu_1)-b(Y_t^{\mu_2,\mu_2},\mu_2)\>\,\d t\\
&\le 2K_1|\Upsilon_t^{\mu_1,\mu_2}|^2\,\d t+2K_3\mathbb W_1(\mu_1,\mu_2)^2\,\d t.
\end{align*}
Thus, Gronwall's inequality yields that  for any $t\in[\si_n,\si_{n+1}),$
\begin{align*}
|\Upsilon_t^{\mu_1,\mu_2}| &\le \big(|\Upsilon_{\si_n}^{\mu_1,\mu_2}| +(2K_3)^{\frac{1}{2}} \mathbb W_1(\mu_1,\mu_2) (t-\si_n)^{\frac{1}{2}}\big)\e^{ K_1(t-\sigma_n)}.
\end{align*}
Subsequently, by the aid of  $\Upsilon_{\si_i}^{\mu_1,\mu_2}=\Upsilon_{\si_i-}^{\mu_1,\mu_2}$, we deduce inductively that for any $t\in[\si_n,\si_{n+1}),$
\begin{align*}
|\Upsilon_t^{\mu_1,\mu_2}| &\le \big(|\Upsilon_{\si_n-}^{\mu_1,\mu_2}| +(2K_3)^{\frac{1}{2}} \mathbb W_1(\mu_1,\mu_2) (t-\si_n)^{\frac{1}{2}}\big)\e^{ K_1(t-\sigma_n)}\\
&\le \e^{K_1t}\big(|\Upsilon_{\si_n-}^{\mu_1,\mu_2}| +(2K_3)^{\frac{1}{2}} \mathbb W_1(\mu_1,\mu_2) t^{\frac{1}{2}}\big) \\
&\le \e^{K_1t}  |\Upsilon_{\si_{n-1}}^{\mu_1,\mu_2}| +(2K_3)^{\frac{1}{2}} \mathbb W_1(\mu_1,\mu_2) t^{\frac{1}{2}} \e^{ 2K_1t}  +(2K_3)^{\frac{1}{2}} \mathbb W_1(\mu_1,\mu_2) t^{\frac{1}{2}}\e^{K_1t}\\
&\le \cdots\\
&\le (2K_3)^{\frac{1}{2}} \mathbb W_1(\mu_1,\mu_2) t^{\frac{1}{2}}\sum_{i=1}^{n+1}\e^{K_1ti}.
\end{align*}
Plugging the previous estimate back into \eqref{S4} enables us to derive that
\begin{equation}\label{WY8}
\begin{split}
\mathbb W_1(\Lambda_{\mu_2}P_t^{\mu_1},\Lambda_{\mu_2}P_t^{\mu_2})&\le (2K_3)^{\frac{1}{2}} \mathbb W_1(\mu_1,\mu_2) t^{\frac{1}{2}}\sum_{n=0}^\8(n+1) \e^{K_1t(1+n)}\P(N_t=n)\\
&\le  \nu({B_1^c})(2K_3)^{\frac{1}{2}} \mathbb W_1(\mu_1,\mu_2) t^{\frac{1}{2}}\e^{(K_1-\lambda^*)t} \sum_{n=0}^\8(n+1) (\e^{K_1t    }\lambda^*t)^n/{n!}\\
&= K_3^{\frac{1}{2}} C_t \mathbb W_1(\mu_1,\mu_2) ,
\end{split}
\end{equation}
where $N_t:=\sharp\{s\in[0,t]: Z^*_s\neq Z^*_{s-}\}$ means the counting process associated with $(Z^*_t)_{t\ge0}$,
and in the second inequality we employ the fact that $\P(N_t=n)=\e^{-\lambda^*t}(\lambda^*t)^n/{n!}$ since $N_t$ is  a Poisson process with the intensity $\lambda^*t$ for $\lambda^*:=\nu( {B_1^c})$. Accordingly,   \eqref{S3} is verifiable.

Let $\pi$ be the unique stationary probability measure of the process $(X_t)_{t\ge0}$.
By virtue of the invariance of $\pi$, it follows from the triangle inequality and \eqref{S1} that
\begin{equation}\label{WY9}
\begin{split}
\mathbb W_1(P_t^*\mu,\pi)&\le   \mathbb W_1((\hat P_t^\mu)^*\mu,(P_t^\pi)^*\mu)+\mathbb W_1( (P_t^\pi)^*\mu,   (P_t^\pi)^*\pi)\\
&\le \mathbb W_1((\hat P_t^\mu)^*\mu,(P_t^\pi)^*\mu)+C\e^{-\lambda t}\mathbb W_1(\mu_1,\mu_2),
\end{split}
\end{equation}
where $(\hat P_t^\mu)^*\mu$ means the law of $X_t^{\mu,\mu}$ determined by (the time-homogeneous version of) \eqref{WY7}.  Next, by repeating the procedure to derive \eqref{WY8},  there exists a locally bounded function $\phi:[0,\8)\to[0,\8)$ such that
\begin{align*}
\mathbb W_1((\hat P_t^\mu)^*\mu,(P_t^\pi)^*\mu)\le K_3^{\frac{1}{2}}\phi(t)\bigg(\int_0^t\mathbb W_1(P_s^*\mu,\pi)^2\d s\bigg)^{\frac{1}{2}}.
\end{align*}
Thus, inserting this back into \eqref{WY9} gives that
 \begin{equation*}
\mathbb W_1(P_t^*\mu,\pi)^2
\le 2 K_3 \phi(t)^2 \int_0^t\mathbb W_1(P_s^*\mu,\pi)^2\d s  +2C^2\e^{-2\lambda t}\mathbb W_1(\mu ,\pi)^2.
\end{equation*}
This, along with Gronwall's inequality, leads to
\begin{align*}
\mathbb W_1(P_t^*\mu,\pi)^2 \le \mathbb W_1(\mu ,\pi)^2\bigg(2C^2\e^{-2\lambda t}+4C^2K_3 \phi(t)^2\int_0^t\e^{-2\lambda s}\e^{2 K_3\int_s^t \phi(u)^2\,\d u }\,\d s\bigg).
\end{align*}
Subsequently, we take $t=t_*:=\frac{1}{2\lambda}\log(1+4C^2)$ and choose $K^*>0$ such that for any $K_3\in(0,K^*],$
\begin{align*}
\lambda^*:=\frac{1}{2}+4C^2K_3 \phi({t_*})^2\int_0^{t_*}\e^{-2\lambda s}\e^{2 K_3\int_s^{t_*} \phi(u)^2\d u }\d s<1.
\end{align*}
As a result, we have
\begin{align*}
\mathbb W_1(P_{t_*}^*\mu,\pi)  \le (\lambda^*)^{\frac{1}{2}}\mathbb W_1(\mu ,\pi).
\end{align*}
Finally, the desired assertion \eqref{WY10} follows from the semigroup property of $(P_t^*)_{t\ge0}$.
\end{proof}

\subsection{Verification of \eqref{UT} in Assumption (${\bf A}_{5}$)}
  For $\hat\mu\in\mathscr P(\R^d)$,
we write  $(\hat \mu P_t^\mu)_{t\ge0}$ as  the distribution of $(Y_t^\mu)_{t\ge0}$ with the initial distribution $\mathscr L_{Y_0^\mu}=\hat\mu$.
In order to  establish  the ergodicity of $(Y_t^\mu)_{t\ge0}$ under the weighted total variation distance,
we assume that
\begin{enumerate}\it
\item[$($${\bf A}_{51}$$)$] there is a constant $\kk \in(0,1]$ such that $J(\kk )>0$, where
 $J(r):=\inf_{x\in\R^d,|x|\le r}\nu_x(\R^d)$ for all $r>0$ with $\nu_x(\d z)=(\nu\wedge(\delta_x*\nu))(\d z)$.

\item[$($${\bf A}_{52}$$)$] for fixed $\mu\in\mathscr P_{\beta^*}^{M^*}$, there is a constant $K:=K_{M^*}>0$ such that for
any $x,y\in\R^d,$
\begin{align*}
\<x-y,b(x,\mu)-b(y,\mu)\>\le K  (|x-y|^2+|x-y|).
\end{align*}

\item[$($${\bf A}_{53}$$)$] $($${\bf A}_1$$)$ with $\theta_1\ge1$ is satisfied.
\end{enumerate}

The following proposition indicates  that $(\hat \mu P_t^\mu)_{t\ge0}$ is weakly contractive under the weighted total variation distance.

\begin{proposition}\label{pro}
Under $({\bf A}_{51})$, $({\bf A}_{52})$ and $({\bf A}_{53})$, for each fixed $\mu\in\mathscr P_{\beta^*}^{M^*}$, there are constants $C_{M^*},\lambda_{M^*}>0$
such that for any $\mu_1,\mu_2\in\mathscr P(\R^d)$ satisfying $\mu_1(V)+\mu_2(V)<\8,$
\begin{align}\label{ER}
\|\mu_1P_t^\mu-\mu_2P_t^\mu\|_{V}\le C_{M^*}\e^{-\lambda_{M^*} t}\|\mu_1-\mu_2\|_{V},
\end{align} where $V(x)=V_\beta(x):=(1+|x|^2)^{\beta/2}$. In particular,
$(Y_t^\mu)_{t\ge0}$ has a unique IPM $\pi^\mu$, and there exists a constant $C_{M^*}^*>0$ so that for all $x\in \R^d$ and $t\ge0$,
\begin{align}\label{EE}
 \|\pi^\mu - \delta_x P_t^\mu  \|_{\rm{var}}\le C_{M^*}^* V(x) \e^{-\lambda_{M^*} t}.
\end{align}
\end{proposition}

\begin{proof}
The proof is in the spirit of that of \cite[Theorem 2.2]{LMW} but with some  modifications. Here we only present the sketch of the proof to highlight the differences. In the sequel, we fix $\mu\in\mathscr P_\beta^{M^*}$. Note that Assumption $($${\bf A}_{1}$$)$ holds with $\theta_1\ge1$.
According to  \eqref{e:add} and  H\"{o}lder's inequality, besides $\theta_1\ge1$ and $\theta_3\theta_4\le \beta^*(1-\gamma_1)$,
it follows  that
\begin{align*}(\mathscr L_\mu V_\beta)(x)\le& -  2^{-\frac{5+\theta_1}{2}} \beta\lambda_1 V_{\beta}(x)   +\Phi(\vv_2,1,l) +  \beta\lambda_2\Gamma(\gamma_1,\gamma_1/\vv_1,\gamma_1)\mu(|\cdot|^{\beta^*})^{\frac{\theta_3\theta_4}{\beta^*(1-\gamma_1)}}. \end{align*}
Below, set $V(x)=V_\beta(x)$ for simplicity. In particular, there are constants $\lambda_V:=\lambda_{V,M^*}>0$ and $ C_V:=C_{V,M^*}>0$
such that
\begin{equation}\label{e:drift}
 (\mathscr L_\mu V)(x)\le -\lambda_V V (x)+C_V,\quad x\in\R^d.
\end{equation}

For notational brevity, we
set
\begin{align*}
S_0:=\{(x,y)\in\R^{2d}:\lambda_V(V(x)+V(y))\le16 C_V  \}.
\end{align*}
Since $\R^d\ni x\mapsto V(x)$ is compact, the quantity $l_0:=1+\sup_{(x,y)\in S_0}|x-y|$ is finite.  In particular, by the definition of $V$ and the fact that $C_V $ is a  positive constant  dependent  on   $M^*$,
$l_0$ is dominated by a constant depending on   $M^*$.

For each $n\ge1$,
define
\begin{align*}
\Phi_n(x,y)=\psi(|x-y|)+a(1-h_n(|x-y|))+\vv(1-h_n(|x-y|))(  V(x)+  V(y)),\quad x,y\in\R^d.
\end{align*}
Herein,  $a:=\frac{8Kc(1+\kk)}{J(\kk)}+\kk^2c^2\e^{-cl_0}$, $\vv:=\frac{1}{16C_V}\kk^2c^2\e^{-cl_0}J(\kappa)$ with  $c:=1+\frac{16Kl_0}{J(\kk)\kk^2}$ and $\kappa,K>0$ being given in Assumptions $($${\bf A}_{51}$$)$ and $($${\bf A}_{52}$$)$ respectively,
\begin{align*}
\psi(r):=
\begin{cases}
1-\e^{-cr},\quad~~~~~~~~~~~~~~~~~~~~  0\le r\le 2l_0,\\
1-\e^{-2cl_0}+\frac{c\e^{-2cl_0}(r-2l_0)}{1+r-2l_0},\quad r\ge 2l_0 ,
\end{cases}
\end{align*}
and
\begin{align*}
h_n(r):=
\begin{cases}
1, \quad~~~~~~~~~~~~~~~~~~~~~~~~~~~~~~~~~~~~~~~~~~~~~~~~~~~~~~~~~~~~~~~   0\le r\le \frac{5}{8n},\\
(4n)^4\Big(r-\frac{7}{8n}\Big)^2\bigg(3\Big(r-\frac{7}{8n}\Big)^2+\frac{2}{n}\Big(r-\frac{7}{8n}\Big)+\frac{3}{8n^2}\bigg),\quad \frac{5}{8n}<r<\frac{7}{8n},\\
0, \quad~~~~~~~~~~~~~~~~~~~~~~~~~~~~~~~~~~~~~~~~~~~~~~~~~~~~~~~~~~~~~~~ r\ge \frac{7}{8n}.
\end{cases}
\end{align*}
Furthermore, we
define the coupling  operator $\tilde{\mathscr L}_\mu$ of $\mathscr L_\mu$  as follows (see \cite[(2.8)]{LMW}): for $f\in C_b^2(\R^{2d})$ and $x,y\in\R^d,$
\begin{align*}
(\tilde{\mathscr L}_\mu f)(x,y)&=\<(\nn_x f)(x,y),b(x,\mu)\>+\<(\nn_y f)(x,y),b(y,\mu)\>\\
&\quad+\frac{1}{2}\int_{\R^d}\big(f(x+z,y+z+(x-y)_\kk)-f(x,y)-\<(\nn_x f)(x,y),z\>\I_{\{|z|\le 1\}}\\
&\qquad\qquad\quad-\<(\nn_y f)(x,y),z+(x-y)_\kk\>\I_{\{|z+(x-y)_\kk|\le 1\}}\big)\nu_{(y-x)_\kk}(\d z)\\
&\quad+\frac{1}{2}\int_{\R^d}\big(f(x+z,y+z+(y-x)_\kk)-f(x,y)-\<(\nn_x f)(x,y),z\>\I_{\{|z|\le 1\}}\\
&\qquad\qquad\quad-\<(\nn_y f)(x,y),z+(y-x)_\kk\>\I_{\{|z+(y-x)_\kk|\le 1\}}\big)\nu_{(x-y)_\kk}(\d z)\\
&\quad+ \int_{\R^d}\big(f(x+z,y+z)-f(x,y)-\<(\nn_x f)(x,y)+(\nn_y f)(x,y),z\>\I_{\{|z|\le 1\}}\big)\\
&\qquad\quad\quad\times\big(\nu-\nu_{(y-x)_\kk}/2-\nu_{(x-y)_\kk}/2\big)(\d z).
\end{align*}

By taking  \cite[Proposition 2.1]{LMW} into consideration,   if  there exists a  constant $\lambda_0>0 $ such that for any
$n>1/\kk$ and
$x,y\in\R^d$ with $\frac{1}{n}\le |x-y|\le n,$
 \begin{equation}\label{EQ}
(\tilde{\mathscr L}_\mu \Phi_n)(x,y)\le-\lambda_0\Phi_n(x,y),
\end{equation}
then, according to the definition of $\Phi_n$,
 we infer that there is a constant $C_0>0$ so that for all $t\ge0$ and $x,y\in\R^d$,
\begin{align*}
\|\delta_xP_t^\mu-\delta_yP_t^\mu\|_V\le C_0\e^{-\lambda_0t}(V(x)+V(y))\I_{\{x\neq y\}}.
\end{align*}
Whereafter,  \eqref{ER} follows by noting that for any $\pi\in\mathscr C(\mu_1,\mu_2)$,
\begin{align*}
\|\mu_1P_t-\mu_2P_t\|_V\le  \int_{\R^d\times\R^d}\|\delta_xP_t^\mu-\delta_yP_t^\mu\|_V\,\pi(\d x,\d y).
\end{align*}

It is easy to see from (${\bf A}_{52}$) and \eqref{e:drift} that for any $x,y\in\R^d$ with $ x\neq y$,
\begin{align*}
(\tilde{\mathscr L}_\mu \Phi_n)(x,y)
&=\frac{1}{2}\nu_{(x-y)_\kk}(\R^d) \Lambda_n(|x-y|) \\
&\quad+\frac{\psi'(|x-y|)}{|x-y|}\<x-y,b(x,\mu)-b(y,\mu)\>+\vv\big((\mathscr L_\mu V)(x)+(\mathscr L_\mu V)(y)\big)\\
&\le \frac{1}{2}\nu_{(x-y)_\kk}(\R^d) \Lambda_n(|x-y|) +K\psi'(|x-y|)(|x-y|+1)\\
&\quad+\vv\big(-\lambda_V V(x)-\lambda_V V(y)+2C_V\big),
\end{align*}
where for any $r\ge0,$
\begin{align*}
\Lambda_n(r):=\psi(r+(\kk\wedge r)) -\psi(r-(\kk\wedge r))-2\psi(r)-ah_n(r-(\kk\wedge r)).
\end{align*}
Note from $\psi(0)=0$ and $h_n(0)=1$ that
\begin{equation*}
\Lambda_n(r)\le
\psi(2r ) -2\psi(r)- a,~0\le r\le \kk;~
\Lambda_n(r)\le\psi(r+\kk) -\psi(r-\kk)-2\psi(r), ~ \kk< r\le l_0.
\end{equation*}
Whence,   \cite[Lemma 2.5]{LMW}  yields  that
\begin{equation*}
\Lambda_n(r)\le
 - a,\quad  0\le r\le \kk; \quad \Lambda_n(r)\le
\psi''(r)\kk^2,\quad     \kk< r\le l_0.
\end{equation*}
Subsequently, by  using  $\psi'(r)=c\e^{-cr}$ and $\psi''(r)=-c^2\e^{-cr}$ for $r\in[0,l_0],$
we derive that
\begin{equation*}
(\tilde{\mathscr L}_\mu \Phi_n)(x,y)\le
\begin{cases}
 -\frac{1}{2}J(\kk) a +Kc (\kk+1) +2C_V\vv-\lambda_V\vv(V(x)+V(y)), \quad\quad\quad~  (x,y)\in \Gamma_{ 1/n,\kk},\\
  -\big(\frac{1}{2}J(\kk)\kk^2c  -2Kl_0 \big)c\e^{-c|x-y|}  +2C_V\vv-\lambda_V\vv(V(x)+V(y)),~   (x,y)\in \Gamma_{\kk,l_0},\\
 2Kl_0\psi'(l_0)-2\vv C_V+\vv\big(-\lambda_V V(x)-\lambda_V V(y)+4C_V\big),\quad\quad\quad~    (x,y)\in\Gamma_{l_0,\8},
\end{cases}
\end{equation*}
where $\Gamma_{r_1,r_2}:=\{(x,y)\in\R^{2d}:r_1\le |x-y|\le r_2\}$ for $0\le r_1<r_2\le\8$.
The estimate above, along with   the definitions of $a$ and $c$,   implies that
\begin{equation*}
(\tilde{\mathscr L}_\mu \Phi_n)(x,y)
\le
\begin{cases}
-\frac{1}{4}J(\kk)\kk^2c\e^{-cl_0}-\lambda_V\vv(V(x)+V(y)), \qquad\qquad  ~  (x,y)\in  S_0\cap\Gamma_{1/n, l_0},\\
-\frac{1}{4}J(\kk)\kk^2c\e^{-cl_0}-\lambda_V\vv(V(x)+V(y)),  \qquad\qquad ~  (x,y)\in  S_0^c\cap\Gamma_{1/n, l_0},\\
-\frac{1}{8}J(\kk)\kk^2c\e^{-cl_0}-\frac{3\vv}{4}\lambda_V(V(x)+V(y)),  \qquad\qquad  (x,y)\in  S_0^c\cap\Gamma_{l_0, \8}.
\end{cases}
\end{equation*}
Therefore, we conclude that
the assertion \eqref{EQ} is valid by taking
\begin{align*}
\lambda_0:=\frac{1}{4} \bigg(\frac{J(\kk)\kk^2c^2\e^{-cl_0}}{ 2(2+a) } \wedge  (3\lambda_V) \bigg).
\end{align*}

Finally, according to \eqref{ER} and the Banach fixed point theorem, we can easily see that
$(Y_t^\mu)_{t\ge0}$ has a unique IPM $\pi^\mu$. Next,  due to the invariance of $\pi^\mu$, it follows from \eqref{ER} that for any $x\in\R^d$ and $\mu\in\mathscr P_{\beta^*}^{M^*}$,
\begin{align*}
\|\pi^\mu-\delta_xP_t^\mu\|_{\rm var} \le C_{M^*}\e^{-\lambda_{M^*}t}(V(x)+\pi^\mu(V)).
\end{align*}
As a result, the assertion \eqref{EE} follows.
\end{proof}

\noindent {\bf Acknowledgements.}\,\,
The research of Jianhai Bao is supported by the National Key R\&D Program of China (2022YFA1006004).
The research of Jian Wang is supported by the National Key R\&D Program of China (2022YFA1006003) and the National Natural Science Foundations of China (No. 12225104).


\begin{thebibliography}{99}
\bibitem{AD}Ahmed, N.U. and  Ding, X.: On invariant measures of nonlinear Markov processes, {\it J. Appl. Math.
Stochastic Anal.}, (1993), 385--406.






\bibitem{BLW}Bao, J., Liu, Y. and Wang, J.: A note on L\'{e}vy-driven McKean-Vlasov SDEs under monotonicity,  arXiv:2412.01070.


\bibitem{BSWX}Bao, J.,  Sun, X.,  Wang, J. and  Xie, Y.: Quantitative estimates for L\'{e}vy driven SDEs with different drifts and applications, {\it J. Differential Equations}, {\bf 398} (2024), 182--217.


\bibitem{BRS} Bogachev, V.I., R\"{o}ckner, M. and  Shaposhnikov, S.V.: 	Convergence in variation of solutions of nonlinear Fokker-Planck-Kolmogorov equations to stationary measures, {\it J. Funct. Anal.}, {\bf 276} (2019), 3681--3713.



\bibitem{Butkovsky}
Butkovsky, O.A.:
On ergodic properties of nonlinear Markov chains and stochastic McKean-Vlasov equations,
{\it Theory Probab. Appl.}, {\bf 58} (2014), 661--674.


\bibitem{Chen}Chen, M.-F.: {\it From Markov Chains to Non-Equilibrium Particle Systems}, 2nd ed. World Scientific, Singapore, 2004.

\bibitem{Cormier} Cormier, Q.:  On the stability of the invariant probability measures of McKean-Vlasov equations, arXiv:2201.11612



\bibitem{DZ} Da Prato, G. and Zabczyk, J.: {\it Ergodicity for Infinite Dimensional Systems}, Cambridge university press, 1996.


\bibitem{Dawson} Dawson, D.A.: Critical dynamics and fluctuations for a mean-field model   of cooperative
behavior, {\it J. Stat. Phys.}, {\bf31} (1983), 29--85.


\bibitem{Deimling} Deimling, K.: {\it Nonlinear Functional Analysis}, Springer-Verlag, Berlin Heidelberg, 1985.

\bibitem{DFG}
Douc, R., Fort, G. and Guillin, A.: Subgeometric rates of convergence of $f$-ergodic strong Markov
processes, {\it Stochastic Process. Appl.},  {\bf 119} (2009),  897--923.

\bibitem{Dra} Dragomir, S.S.: \emph{Some Gronwall Type Inequalities and Applications}, Nova Science Publishers, Inc., Hauppauge, NY, 2003.


\bibitem{Eb} Eberle, A.:  Reflection couplings and contraction rates for diffusions, {\it  Probab. Theory Related Fields}, {\bf 166} (2016),   851--886.


\bibitem{EGZ}Eberle, A.,  Guillin, A. and  Zimmer, R.:  Quantitative Harris-type theorems for diffusions and McKean-Vlasov processes, {\it  Trans. Amer. Math. Soc.}, {\bf 371} (2019),   7135--7173.



\bibitem{HM}  Hairer, M. and  Mattingly,  J. C.: Yet another look at Harris's ergodic theorem for Markov chains. In: {\it Seminar on Stochastic Analysis, Random
Fields and Applications VI},  109--117, Progr. Probab. 63, Birkh\"{a}user/Springer, Basel AG, Basel, 2011.


\bibitem{HMS}Hairer, M.,  Mattingly, J.C. and  Scheutzow, M.:  Asymptotic coupling and a general form of Harris' theorem with applications to stochastic delay equations,  {\it Probab. Theory Related Fields},  {\bf149} (2011),   223--259.


\bibitem{Hairer} Hairer, M., Convergence of Markov Processes, https://www.hairer.org/notes/Convergence.pdf.


\bibitem{HMW}Huang, L.-J.,  Majka, Mateusz B. and  Wang, J.: Approximation of heavy-tailed distributions via stable-driven SDEs, {\it Bernoulli}, {\bf 27} (2021),   2040--2068.



 \bibitem{LS}
Lazi\'{c}, P. and Sandri\'{c}, N.: On sub-geometric ergodicity of diffusion processes, {\it Bernoulli}, {\bf 27} (2021), 348--380.


 \bibitem{LMW}Liang, M.,  Majka, Mateusz B. and  Wang, J.:  Exponential ergodicity for SDEs and McKean-Vlasov processes with L\'{e}vy noise, {\it Ann. Inst. Henri Poincaré Probab. Stat.}, {\bf 57} (2021),   1665--1701.


\bibitem{LWZ}Liu, W.,  Wu, L. and  Zhang, C.:  Long-time behaviors of mean-field interacting particle systems related to McKean-Vlasov equations, {\it Comm. Math. Phys.}, {\bf 387} (2021),  179--214.


\bibitem{LXZ}Liu, Y.,  Xie, Y. and  Zhang, M.: Phase transitions for a class of time-inhomogeneous diffusion processes, {\it J. Stat. Phys.}, {\bf 190} (2023),  Paper no. 42, 16 pp.

\bibitem{LWa}Luo, D. and  Wang, J.:  Exponential convergence in $L^p$-Wasserstein distance for diffusion processes without uniformly dissipative drift, {\it Math. Nachr.}, {\bf 289} (2016),   1909--1926.


 \bibitem{LW}Luo, D. and  Wang, J.:  Refined basic couplings and Wasserstein-type distances for SDEs with L\'evy noises, {\it Stochastic Process. Appl.}, {\bf 129} (2019), 3129--3173.



 \bibitem{MT}
Meyn, S.P. and Tweedie, R.L.: Stability of Markovian processes. III. Foster–Lyapunov criteria for
continuous-time processes, {\it Adv. Appl. Probab.}, {\bf25} (1993), 518--548.


\bibitem{MR} Monmarch\'{e}, P.  and  Reygner, J.: Local convergence rates for Wasserstein gradient flows
and McKean-Vlasov equations with multiple stationary
solutions,  arXiv:2404.15725.




\bibitem{PR}Pr\'{e}v\^ot, C. and   R\"ockner, M.:
\emph{A Concise Course on Stochastic Partial Differential Equations}, Lecture Notes in Mathematics, vol. {\bf 1905},  Springer, Berlin, 2007.




\bibitem{QYZ}Qu, B., Yao, J. and Zhi, Y.: Global dynamics of McKean-Vlasov SDEs via stochastic order, arXiv: 2411.01364.


\bibitem{Rachev} Rachev, S.: The Monge-Kantonovich mass transference problem and its stochastic applications, {\it Theory Probab. Appl.}, XXIX (1985), 647--676.





\bibitem{Tugautb}  Tugaut, J.:  Convergence to the equilibria for self-stabilizing processes in double well landscape, {\it Ann. Probab.}, {\bf 41} (2010), 1427--1460.




\bibitem{Tugaut} Tugaut, J.: Phase transitions of McKean-Vlasov processes in double-wells landscape, {\it Stochastics},
{\bf86} (2014), 257--284.


\bibitem{Tugautc}  Tugaut, J.:  Convergence in Wasserstein distance for self-stabilizing
diffusion evolving in a double-well landscape, {\it C. R. Acad. Sci. Paris, Ser. I}, {\bf 356} (2018), 657--660.



\bibitem{Wang}Wang, F.-Y.:
Distribution dependent SDEs for Landau type equations,
{\it Stochastic Process. Appl.}, {\bf 128} (2018),  595--621.

\bibitem{Wangb}Wang, F.-Y.: Exponential ergodicity for non-dissipative McKean-Vlasov SDEs, {\it Bernoulli}, {\bf 29} (2023),   1035--1062.

\bibitem{WW} Wang, F.-Y. and  Wang, J.:  Harnack inequalities for stochastic equations driven by L\'evy noise, {\it J. Math. Anal. Appl.},  {\bf410} (2014),   513--523.

\bibitem{Wang2010} Wang, J.: Regularity of semigroups generated by Levy type
operators via coupling, {\it Stochastic Process. Appl.}, {\bf 120} (2010), 1680--1700.




\bibitem{XZ}Xie, L. and Zhang, X.:
Ergodicity of stochastic differential equations with jumps and singular coefficients,
{\it Ann. Inst. Henri Poincar\'{e} Probab. Stat.}, {\bf 56} (2020),  175--229.


\bibitem{Y}
Yan, J.-A.: A perturbation theorem for semigroups of linear operators, in: \emph{S\'{e}minaire de Probabilit\'{e}s XXII}, Lecture Notes in Mathematics, vol. {\bf 1321}, Springer, 1988, 89--91.

\bibitem{Zhang} Zhang, S.-Q.:  Existence and non-uniqueness of stationary distributions for distribution dependent SDEs, {\it Electron. J. Probab.},
    {\bf 28} (2023), Paper No. 93, 34 pp.

\bibitem{Zhangb} Zhang, S.-Q.: Local convergence near equilibria for distribution dependent SDEs, arXiv:2501.04313.



\bibitem{Zhangxicheng}Zhang, X.:
Derivative formulas and gradient estimates for SDEs driven by $\alpha$-stable processes,
{\it Stochastic Process. Appl.},  {\bf123} (2013),   1213--1228.


\end{thebibliography}
\end{document}